\documentclass[11pt,twoside,letter]{article}
\usepackage[dvips]{graphicx}
\usepackage{fancyhdr}
\usepackage{bm}

\def\@begintheorem#1#2{
   \par\noindent\bgroup{\sc #1\ #2. }\it\ignorespaces}
\def\@beginremark#1#2{
   \par\noindent\bgroup{\sc #1\ #2. }\ignorespaces}
\def\@opargbegintheorem#1#2#3{
   \par\bgroup{\sc #1\ #2\ (#3). }\it\ignorespaces}
\def\@endtheorem{\egroup}

\def\shah{{\makebox[2.7ex][s]{$\sqcup$\hspace{-0.5em}\hfill$\sqcup$}}}
\def\shah1{{\makebox[2.3ex][s]{$\sqcup$\hspace{-0.15em}\hfill $\sqcup$}}}

\newtheorem{theorem}{Theorem}[section]
\newtheorem{lemma}[theorem]{Lemma}
\newtheorem{proposition}[theorem]{Proposition}
\newtheorem{definition}[theorem]{Definition}
\newtheorem{remark}[theorem]{Remark}
\newtheorem{example}[theorem]{Example}

\newtheorem{corollary}[theorem]{Corollary}
\newtheorem{question}[theorem]{Question}

\newcommand{\Sc}{\mathcal S}
\usepackage{amssymb}
\usepackage{verbatim}
\usepackage{amsmath}
\usepackage{rotating}
\usepackage[pdftex]{color}

\usepackage{rawfonts}
\input prepictex
\input pictexwd
\input postpictex

\newcommand{\commentout}[1]{}

\newcommand{\fhat}[1]{\widehat{#1}}
\newcommand{\Z}{\mathbb{Z}}

\newcommand{\supp}{{\rm supp}\,}

\newcommand{\R}{\mathbb{R}}

\newcommand{\N}{\mathbb{N}}

\newcommand{\C}{\mathbb{C}}

\newcommand{\eqa}[1]{\begin{eqnarray}#1 \end{eqnarray}}
\newcommand{\Tg}{\Phi_g}

\newcommand{\calS}{{\cal S}}

\newcommand{\vol}{ {\rm vol } }

\newcommand{\ga}{ {\mathfrak g} }
\newcommand{\Ga}{ {\mathfrak G} }
\newcommand{\lllambda}{u}
\newcommand{\weight}{w}

\newcommand{\Rh}{{\R}}

\newcommand{\rone}{x}
\newcommand{\rtwo}{t}
\newcommand{\rthree}{x'}

\newcommand{\gone}{\xi}
\newcommand{\gtwo}{\nu}
\newcommand{\gthree}{\xi'}
\newcommand{\gfour}{\nu'}

\def\supp{\mathop{\textstyle{\rm supp}}\nolimits}

\begin{document}

%%%%%%%%%%%%%%%%%%%%%%%%%%%%%%%%%%%%%%%%%%%%%%%%%%
%
% The Title
%
\title{\bf\vspace{-39pt}
Sampling of Operators}
%
%%%%%%%%%%%%%%%%%%%%%%%%%%%%%%%%%%%%%%%%%%%%%%%%%%%
%%
% The Authors
%

\author{G\"otz E. Pfander \\ \small School of Engineering and Science, Jacobs University, \\
\small 28759 Bremen, Germany \\ \small g.pfander@jacobs-university.de}

%%%%%%%%%%%%%%%%%%%%%%%%%%%%%%%%%%%%%%%%%%%%%%%%%%%
%%
%% Do not print the date
%%
%\date{}
%
%%%%%%%%%%%%%%%%%%%%%%%%%%%%%%%%%%%%%%%%%%%%%%%%%%%
%%
%% Make the title and set the header styles to be
%%   fancy for this page.  STSIP will adjust these
%%   headings later.
%
\maketitle
\thispagestyle{fancy}

%%%%%%%%%%%%%%%%%%%%%%%%%%%%%%%%%%%%%%%%%%%%%%%%%%
%
% Setup the Headings for the article
%   - Please use all caps and initials for your
%     first and middle names
%
\markboth{\footnotesize \rm \hfill G.E. PFANDER
\hfill} {\footnotesize \rm \hfill Sampling  of Operators\hfill}

%%%%%%%%%%%%%%%%%%%%%%%%%%%%%%%%%%%%%%%%%%%%%%%%%%
%
% The Abstract, keywords, and phrases
%
\begin{abstract}
Sampling and reconstruction of functions is a central tool in science. A key result is given by the sampling theorem for bandlimited functions attributed to Whittaker, Shannon, Nyquist, and Kotelnikov.  We develop an analogous sampling theory for operators which we call bandlimited if their Kohn-Nirenberg  symbols are bandlimited.  We prove sampling theorems for such operators and show that they are extensions of the classical sampling theorem.
\end{abstract}
\maketitle
%%%%%%%%%%%%%%%%%%%%%%%%%%%%%%%%%%%%%%%%%%%%%%%%%%%%%%%%%%%%%%%%%%%%%%%%%%%%%%%%%%%

\section{Introduction}\label{section:introduction}

The classical sampling theorem for bandlimited functions states that a function whose Fourier
transform is supported on an interval of length $\Omega$ is completely
characterized by samples taken at rate at least $1/\Omega$ per unit interval.
That is, with $\mathcal F$ denoting the Fourier transform\footnote{See Section~\ref{section:notation} for basic notation used throughout this paper.} we have

\begin{theorem}\label{thm:shannon}
  For $f\in L^2(\R)$ with $\supp \mathcal F f \subseteq [-\frac \Omega 2, \frac \Omega 2)$, choose $T$ with  $T\Omega\leq 1$. Then
  \begin{eqnarray}
    \|\{f(nT)\}\|_{l^2(\Z)}=T\, \|f\|_{L^2(\R)},\notag %\label{eqn:shannonstability}
  \end{eqnarray} and $f$ can be reconstructed by means of
  \begin{eqnarray}
    f(x)=\sum_{n\in\Z} f(nT)\, \frac{\sin(\pi T (x-n))}{\pi T (x-n)} \notag % \label{eqn:shannonreconstruction}
  \end{eqnarray}
  with convergence in $L^2(\R)$.
\end{theorem}

Theorem~\ref{thm:main-simple} is an exemplary result from our sampling theory of operators. We choose a Hilbert--Schmidt operator $H$ on $L^2(\R)$ with kernel $\kappa_H$ and Kohn-Nirenberg symbol $\sigma_H$, that is $\sigma_H(x,D)=H$ in the sense of pseudodifferential operators \cite{Hor79,Tay81}. Recall that Hilbert--Schmidt operators on $L^2(\R)$ are exactly those bounded operators $H$  with $\sigma_H\in L^2(\R^2)$ and corresponding norm of $H$. Let $\mathcal F^s$ denote the so-called symplectic Fourier transform$^2$  on $L^2(\R^{2d})$.

\begin{theorem}\label{thm:main-simple}
  For $H: L^2(\R)\longrightarrow L^2(\R)$ Hilbert--Schmidt with $\supp \mathcal F^s \sigma_H \subseteq[0, T) {\times}[-\frac \Omega 2, \frac \Omega 2)$ and  $T\Omega\leq 1$, we have
  \begin{eqnarray}
    \|H\sum_{k\in\Z}\delta_{kT}\|_{L^2(\R)}=T\|H\|_{HS},\notag %\label{eqn:operatorstability-simple}
  \end{eqnarray} and  $H$ can be reconstructed by means of
  \begin{eqnarray}
    \kappa_H(x+t,x)=\sum_{n\in\Z} \big(H\sum_{k\in\Z}\delta_{kT}\big)(t+nT)\, \frac{\sin(\pi T (x-n))}{\pi T (x-n)} \notag%\label{eqn:operatorreconstruction-simple}
  \end{eqnarray}
  with convergence in $HS(L^2(\R^2)$.
\end{theorem}

As shown in Section~\ref{section:MainResults1}, Theorem~\ref{thm:shannon} can be deduced from  the general form of Theorem~\ref{thm:main-simple} which is stated below as Theorem~\ref{thm:main-full}.

The appearance of the sampling rate $T$ in the description of the bandlimitation of the operator's Kohn--Nirenberg symbol reflects a fundamental difference between the sampling of operators and the sampling of functions.  This phenomenon is illuminated  in terms of operator identification by Theorem 3.6 in \cite{KP06} and Theorem 1.1 in \cite{PW06b}, results which we extend here by Theorems~\ref{thm:main-full2} and \ref{thm:main-full3} below. In fact, in the classical sampling theory,  the bandlimitation of a function to a large interval can be compensated by choosing a correspondingly high sampling rate.  In the here developed sampling theory for operators, only bandlimitations to sets of area less than or equal to one permit sampling and reconstruction. The bandlimitation to, for example, a rectangle of area 2 cannot be compensated by increasing the sampling rate, and, in fact, operators characterized by such a bandlimitation cannot be determined in a stable manner by the application of the operator to a single function or distribution, whether it is supported on a discrete set (which we shall refer to as sampling set below), or not.

Theorems~\ref{thm:main-full2} and \ref{thm:main-full3} in simple terms is Theorem~\ref{thm:main-simple2}  below. It can also be deduced from earlier operator identification results in  \cite{Pfa08,PW06}.
As it is customary to define Paley--Wiener spaces
\begin{eqnarray}
  PW\big(M \big)=\{f\in L^2(\R^d):\, \supp \mathcal F{f} \subseteq M \} \notag
\end{eqnarray}
to describe spaces of  functions bandlimited to $M\subseteq \R^d$, we introduce in this paper operator Paley--Wiener spaces
\begin{eqnarray}
  OPW\big(M \big)=\{H\in HS(L^2(\R^d)):\, \supp \mathcal F^s \sigma_H \subseteq M \} \notag
\end{eqnarray}
to describe operators bandlimited to $M\subseteq \R^{2d}$. In short,  $PW(M)$ and $OPW(M)$ are linked via the Kohn-Nirenberg correspondence \cite{Fol89,KN65}.

\begin{theorem}\label{thm:main-simple2} Let $\mu(M)$ denote the Lebesgue measure of the set $M\subseteq \R^2$.
\begin{enumerate}
  \item For $M$  compact with $\mu(M)<1$ exists $T>0$,  a bounded sequence $\{c_k\}$, and $A,B> 0$ with
  \begin{eqnarray}
        A \|H\|_{HS} \leq \|H\sum_{k\in\Z}c_k \delta_{kT}\|_{L^2(\R)} \leq B \|H\|_{HS},\quad H\in OPW(M).\notag
      \end{eqnarray}
      \item Let $M$ be open with $\mu(M)>1$, then, for all $g\in \mathcal S'(\R)$ and $\epsilon>0$, exists $H\in OPW(M)$ with
        \begin{eqnarray}
        \|Hg\|_{L^2(\R)} \leq \epsilon \|H\|_{HS}. \notag
      \end{eqnarray}
\end{enumerate}
\end{theorem}

The sampling theory developed here has roots in the work of Kozek, Pfander \cite{KP06} and Pfander, Walnut \cite{PW06b} which addressed the identifiability of slowly time--varying operators, that is, of so--called underspread operators.  Measurability or identifiability  of a given operator class describes the property that all operators of that class can be distinguished by their action on a well chosen single function or distribution. The importance of operator identification and, therefore, operator sampling in engineering and science is illustrated by the following two examples. In case of information transmission, complete knowledge of the communications channel operator at hand allows the transmitter to optimize its transmission strategy in order to transmit information close to channel capacity (see, for example, \cite{Gol05} and references therein). In radar, simply speaking, a signal is send out and the goal is to determine the nature of reflecting objects from the received echo, that is, from the  response to the radar channels input signal \cite{Sko80,KP06}.

The best known operator identification example states that time--invariant operators are fully characterized by their response to a Dirac impulse. Kailath \cite{Kai62} and later Bello \cite{Bel69} investigated the identifiability of slowly time varying channels (operators) which are defined by the support size of their spreading functions, namely of the symplectic Fourier transform of the operators' Kohn--Nirenberg symbols. In both papers, conjectures were made that were then proven in \cite{KP06}, respectively \cite{PW06b}.   The operator sampling Theorems~\ref{thm:main-full2} and \ref{thm:main-full3} extend the main  results in \cite{KP06,PW06b}, Shannon's sampling theorem, as well as the fact that time--invariant operators are identifiable by their impulse response (see Figure~\ref{fig:niklasPic}).

\begin{figure}
\begin{center}
  \includegraphics[height=14cm]{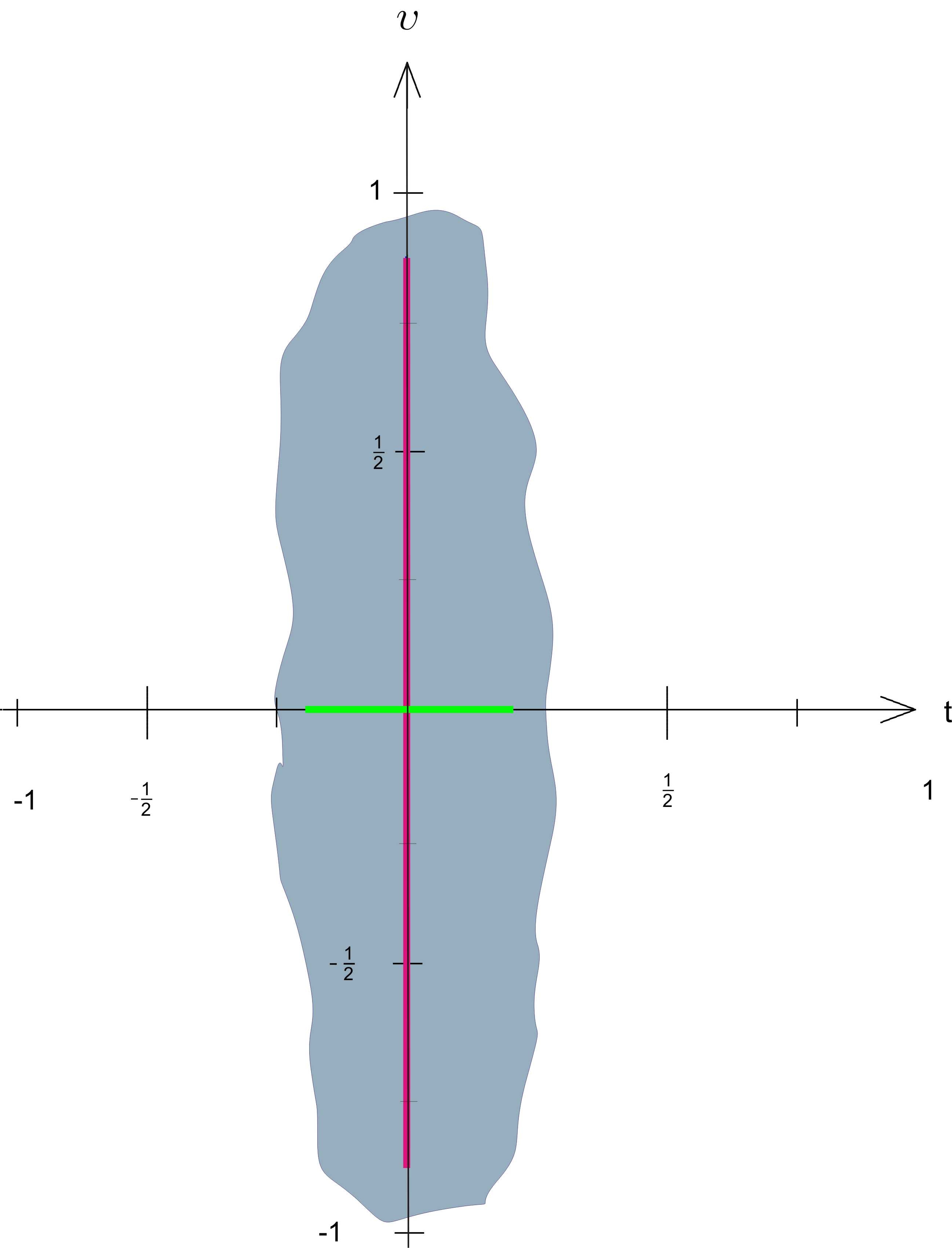}
\end{center}

\caption{In the one dimensional case, the herein developed  sampling theory for operators applies to any pseudodifferential operators whose Kohn--Nirenberg symbol is bandlimited to a compact set of Lebesgue measure less than one (for example, the blue  region above). The results extend  the classical sampling theorem described in Theorem~\ref{thm:shannon} which is equivalent to the identifiability of operators whose Kohn-Nirenberg symbol is bandlimited to a segment of the frequency shift axis (red).  Also, the fact that time--invariant operators with compactly supported impulse response can be identified from their action on the Dirac   impulse  is a special case of our results since the Kohn--Nirenberg symbols of    time--invariant operators  are bandlimited to the time shift axis (green).\label{fig:niklasPic}}
\end{figure}

The paper is structured as follows.  Section~\ref{section:notation} provides background on time--frequency analysis of functions and distributions, in particular on modulation spaces (Section~\ref{section:modulationspaces}), as well as on time--frequency analysis of pseudodifferential operators (Section~\ref{section:pseudodifferentialoperators}). In Section~\ref{section:combine} we discuss boundedness of pseudodifferential operators on modulation spaces. In Sections~\ref{section:MainResults1} and \ref{section:MainResults2}  we state and prove our main results. Section~\ref{section:outlook} contains references to recent progress  and open questions in the sampling theory for operators .

\section{Background}\label{section:notation}

$L^2(\R^d)$ denotes the Hilbert space of complex valued, Lebesgue measurable functions on Euclidean space $\R^d$  \cite{Fol99}. The {\it Fourier transformation} $\mathcal F$, respectively the {\it symplectic Fourier transformation}  $\mathcal F_s$, is the unitary operator
$
\mathcal F : L^2(\R^d) \longrightarrow L^2(\R^d),\quad f\mapsto  \widehat f = \mathcal Ff
$,
 densely defined by
$$
\widehat f(\gamma)=\int_{\R^d} f(x)\, e^{-2\pi i \gamma \cdot x}\, dx,\quad f\in L^1(\R^d){\cap} L^2(\R^d) ,
$$
respectively
 $\mathcal F^s:L^2(\R^{2d})\longrightarrow L^2(\R^{2d}) $
with
$$
\mathcal F^s F (t,\nu){=}\iint_{\R^{2d}} F(x,\xi)\, e^{-2\pi i [(t,\nu), (x,\xi)]}\, dx\,d\xi {=}\iint_{\R^{2d}} F(x,\xi)\, e^{-2\pi i (\nu \cdot x {-} \xi\cdot t)}\, dx\,d\xi ,\ \  F{\in} L^1(\R^{2d}){\cap} L^2(\R^{2d}) ,
$$
where $[\cdot,\cdot]$ denotes the symplectic form on $\R^{2d}$. Throughout the paper, integration is with respect to the Lebesgue measure which we denote by $\mu$.

The Fourier transform defines isomorphisms on the Frechet space of Schwartz functions $\Sc(\R^d)$ and on its dual $\Sc'(\R^d)$ of tempered distributions (equipped with the
weak-$\ast$ topology).  Note that $\Sc'(\R^d)$ contains constant functions, {\em Dirac's delta}
$\delta:f\mapsto f(0)$, and weighted {\em Shah distributions}
$\sum_{n\in\Z^d}c_n \delta_{kT}$,  $T\in(\R^+)^d$,
with $\{c_n\}$ having at most polynomial growth.

Similarly to the {\it Fourier transformation}, the {\em time shift operator}
$T_t$, $t\in\R^d$, given by $T_tf(x)=f(x-t)$ and the {\em modulation operator} $M_\weight$, $\weight \in \R^d$, $M_\weight f(x)=e^{2\pi i \weight{\cdot} x } f(x)$, act as unitary operators on $L^2(\R)$ and they are  isomorphisms on $\Sc(\R^d)$ and $\Sc'(\R^d)$. Note that $M_\weight$ is also called {\em frequency shift operator} since $\widehat{M_\weight f}=T_\weight \widehat f$. Further, we refer to $\pi(\lambda)=\pi(t,\nu)=M_\nu T_t$ for $\lambda=(t,\nu)\in\R^{2d}$ as {\em time--frequency shift operator}. Note that we have $ \mathcal F \circ \pi(t,\nu) =  e^{2\pi i t\nu}\pi(\nu,-t) \circ \mathcal F$, that is, $ \mathcal F \pi(t,\nu) f=  e^{2\pi i t\nu}\pi(\nu,-t) \widehat f$ for $f\in \mathcal S'(\R^d)$. %For simplicity of notation, we shall also write
%$\widehat{\pi}(\lambda)=\widehat{\pi}(t,\nu)=e^{2\pi i t\nu}\pi(\nu,-t)$.

The goal of {\em operator  identification} is to select, for given  spaces $X$ and $Y$ of functions or distributions defined on $\R^d$ and a given space of  linear operators $\mathcal{H}$ mapping $X$ to $Y$, an element $g\in X$ which induces a continuous, open, and injective map $\Phi_g: {\mathcal{H}}\longrightarrow Y(\R^d),\ H\mapsto Hg$
(see Figure~\ref{figure:identification}).

\begin{definition}\label{def:operatoridentification}
Let $X$ be a set, $Y$  a topological vector space, and
$\mathcal Z$ a topological vector space of operators mapping $X$ to $Y$. The space $\mathcal Z$ is {\em identifiable by $g\in
X$} if $\Phi_g: {\mathcal Z }\longrightarrow Y,\ H\mapsto Hg$ is continuous, open and injective.  In the case that $Y$ and $\mathcal H$ are normed spaces, this reads: there exist $A,B>0$ with
\begin{equation}\label{eqn:identifiable}
    A \, \|H\|_{\mathcal{Z}}
        \leq
            \|Hg\|_Y
        \leq
            B \, \|H\|_{\mathcal{Z}},\quad H\in{\mathcal{Z}}\,.
\end{equation}
If we can
choose  $g \in
X=X(\R^d)$ of the form $g=\sum_j c_j \delta_{x_j}$, $x_j\in \R^d$ and $c_j\in\C$ for $j\in \Z^d$, as identifier,  then we say that $\mathcal Z$ mapping $X$ to $Y$ permits operator sampling and we call
$\{x_j\}$ a set of sampling  for $\mathcal Z$ with respective  sampling weights $\{c_j\}$. We refer to $g$ as a sampling function for the operator class $\mathcal Z$.
\end{definition}

In the following, we shall abbreviate norm equivalences as the one given in \eqref{eqn:identifiable} using the symbol $\asymp$. For example, \eqref{eqn:identifiable} becomes
\begin{equation}\notag % \label{eqn:identifiable2}
  \|H\|_{\mathcal{Z}}
        \asymp
            \|Hg\|_Y\,, \quad  H\in{\mathcal{Z}}\,.
\end{equation}

\begin{figure}
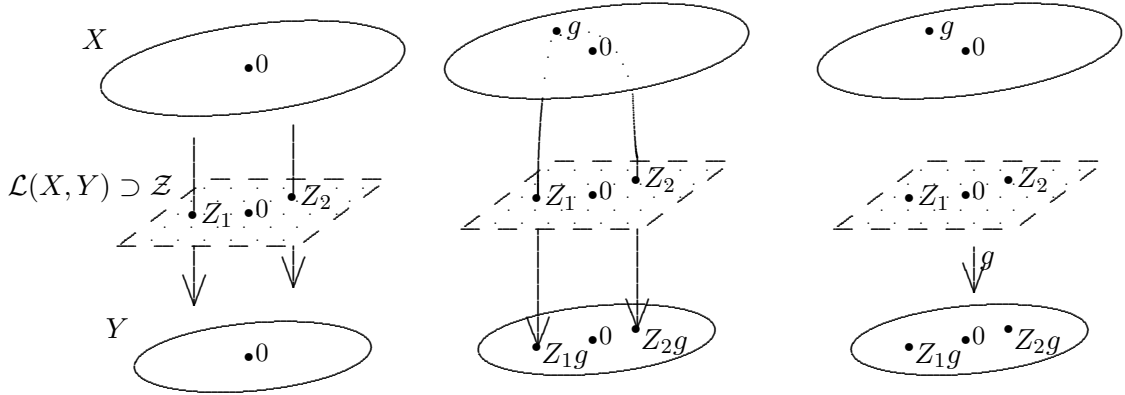
\label{figure:identification}
\begin{minipage}[c][8cm][c]{7cm}
\[
\beginpicture
\setcoordinatesystem units <.6cm,.6cm>
%
% X (position of center 4 6.75)
\startrotation by 0.99 0.13 about 4 6.75 \ellipticalarc axes ratio
3.5:1 360 degrees from 1.2 6.2 center at 4 6.75 \stoprotation \put
{$X$} at 0.5 7.4 \put {\small 0} at 4.2 6.85 \put {\circle*{3}} at 4
6.75
%
% Z  (position of center 4 3.75)
\setdashes <0.25cm> \setlinear \plot 1 2.8  5 2.8  7 4.3  3 4.3  1
2.8 / \hshade 2.8 .8 5 4.3 3 7 / \put {$ {\mathcal L}(X,Y)\supset
{\mathcal Z}$} at 0.4 4.1 \put {\small 0} at 4.2 3.65 \put
{\circle*{3}} at 4 3.55
%
% Y  (position of center 4 0.75)
\setsolid \startrotation by 0.99 0.09 about 4 0.35 \ellipticalarc
axes ratio 3.5:1 360 degrees from 1.5 0.1 center at 4 0.35
\stoprotation \put {$Y$} at 1 1 \put {\small 0} at 4.2 0.45 \put
{\circle*{3}} at 4 0.35
%
%arrows
%Z1
\plot 2.7 5.2  2.7 3.5 / \put {\circle*{3}} at 2.75 3.5 \arrow
<0.4cm> [0.375,0.75]     from 2.7 2.8 to 2.7 1.5 \put {$Z_1$} at 3.2
3.5
%Z2
\plot 4.9 5.5  4.9 3.9 / \put {\circle*{3}} at 4.95 3.9 \arrow
<0.4cm> [0.375,0.75]     from 4.9 2.8 to 4.9 1.9 \put {$Z_2$} at 5.4
3.9
\endpicture
\]
\end{minipage}
%%%%%%%%%%%%%%%%%%%%%%%%%%%%%%%%%%%%%%%%%%%%%%%%%%%%%%%%%%%%%%%%%%%%%%%%%%%%%%%%%%%
%
%
% Second scheme
%
%%%%%%%%%%%%%%%%%%%%%%%%%%%%  Minipage  %%%%%%%%%%%%%%%%%%%%%%%%%%%%%%%%%%%%%%%%%%%
\begin{minipage}[c][8cm][c]{4.5cm}
\[
\beginpicture
\setcoordinatesystem units <.6cm,.6cm>
%
% x (position of center 4 6.75)
\startrotation by 0.99 0.13 about 4 6.75 \ellipticalarc axes ratio
3.5:1 360 degrees from 1.2 6.2 center at 4 6.75 \stoprotation
%\put {$X$} at 8.1 6.6
\put {\small 0} at 4.2 6.85 \put {\circle*{3}} at 4 6.75 \put {$g$}
at 3.5 7.2 \put {\circle*{3}} at 3.2 7.2
%
% z  (position of center 4 3.75)
\setdashes <0.25cm> \setlinear \plot 1 2.8  5 2.8  7 4.3  3 4.3  1
2.8 / \hshade 2.8 .8 5 4.3 3 7 /
%\put {${\mathcal Z}$} at 8 3.4
\put {\small 0} at 4.2 3.65 \put {\circle*{3}} at 4 3.55
%
% y (position of center 4 0.75)
\setsolid \startrotation by 0.99 0.09 about 4 0.35 \ellipticalarc
axes ratio 3.5:1 360 degrees from 1.5 0.1 center at 4 0.35
\stoprotation
%\put {$Y$} at 7.5 0.2
\put {\small 0} at 4.2 0.45 \put {\circle*{3}} at 4 0.35
%
%arrows
%Z1
\setquadratic \setdots \plot 3.2 7.1 3 7 2.8 5.7 / \setsolid \plot
2.8 5.7 2.72 4.6 2.7 3.5 / \setlinear \plot 2.7 4  2.7 3.5 / \put
{\circle*{3}} at 2.75 3.5 \arrow <0.4cm> [0.375,0.75]     from 2.7
2.8 to 2.7 0.2 \put {$Z_1$} at 3.2 3.5 \put {$Z_1g$} at 3.3 0 \put
{\circle*{3}} at 2.75 0.2
%Z2
\setquadratic \setdots \plot 3.2 7.1 4.2 7 4.8 5.8 / \setsolid \plot
4.8 5.8 4.85 4.8 4.9 4.4 / \setlinear \plot 4.9 4.4  4.9 3.9 / \put
{\circle*{3}} at 4.95 3.9 \arrow <0.4cm> [0.375,0.75] from 4.9 2.8
to 4.9 0.6 \put {$Z_2$} at 5.4 3.9 \put {$Z_2g$} at 5.5 0.3 \put
{\circle*{3}} at 4.95 0.6
\endpicture
\]
\end{minipage}
%%%%%%%%%%%%%%%%%%%%%%%%%%%%%%%%%%%%%%%%%%%%%%%%%%%%%%%%%%%%%%%%%%%%%%%%%%%%%%%%%%%
%
%
% Third scheme
\hspace{0.2cm}
%
%%%%%%%%%%%%%%%%%%%%%%%%%%%%  Minipage  %%%%%%%%%%%%%%%%%%%%%%%%%%%%%%%%%%%%%%%%%%%
\begin{minipage}[c][8cm][c]{4.5cm}
\[
\beginpicture
\setcoordinatesystem units <.6cm,.6cm>
%
% x (position of center 4 6.75)
\startrotation by 0.99 0.13 about 4 6.75 \ellipticalarc axes ratio
3.5:1 360 degrees from 1.2 6.2 center at 4 6.75 \stoprotation
%\put {$X$} at 8.1 6.6
\put {\small 0} at 4.2 6.85 \put {\circle*{3}} at 4 6.75 \put {$g$}
at 3.5 7.2 \put {\circle*{3}} at 3.2 7.2
%
% z  (position of center 4 3.75)
\setdashes <0.25cm> \setlinear \plot 1 2.8  5 2.8  7 4.3  3 4.3  1
2.8 / \hshade 2.8 .8 5 4.3 3 7 /
%\put {${\mathcal Z}$} at 8 3.4
\put {\small 0} at 4.2 3.65 \put {\circle*{3}} at 4 3.55
%
% y (position of center 4 0.75)
\setsolid \startrotation by 0.99 0.09 about 4 0.35 \ellipticalarc
axes ratio 3.5:1 360 degrees from 1.5 0.1 center at 4 0.35
\stoprotation
%\put {$Y$} at 7.5 0.2
\put {\small 0} at 4.2 0.45 \put {\circle*{3}} at 4 0.35
%
%arrows
%Z1
%\setquadratic
%\setdots
%\plot 3.2 7.1 3 7 2.8 5.7 /
%\setsolid
%\plot 2.8 5.7 2.72 4.6 2.7 3.5 /
%\setlinear
%\plot 2.7 4  2.7 3.5 /
\put {\circle*{3}} at 2.75 3.5
%\arrow <0.4cm> [0.375,0.75]     from 2.7 2.8 to 2.7 0.2
\put {$Z_1$} at 3.2 3.5 \put {$Z_1g$} at 3.3 0 \put {\circle*{3}} at
2.75 0.2
%Z2
%\setquadratic
%\setdots
%\plot 3.2 7.1 4.2 7 4.8 5.8 /
%\setsolid
%\plot 4.8 5.8 4.85 4.8 4.9 4.4 /
%\setlinear
%\plot 4.9 4.4  4.9 3.9 /
\put {\circle*{3}} at 4.95 3.9
%\arrow <0.4cm> [0.375,0.75]     from 4.9 2.8 to 4.9 0.6
\put {$Z_2$} at 5.4 3.9 \put {$Z_2g$} at 5.5 0.3 \put {\circle*{3}}
at 4.95 0.6
%
% one arrow
\arrow <0.4cm> [0.375,0.75]     from 4.1 2.4 to 4.1 1.4 \put {$g$}
at 4.4 2.1
\endpicture
\]
\end{minipage}

\caption{Illustration of the operator identification and sampling problem.
We seek an element $g\in X$ in the domain of the operator class $\mathcal Z$
which induces a map from $\mathcal Z$ into the range space $Y$ which is continuous, open, and injective. If we can choose $g=\sum_j c_j \delta_{x_j}$, then $\mathcal Z$ permits operator sampling.}
\end{figure}

We shall describe in Section~\ref{section:modulationspaces} the distribution spaces and in Section~\ref{section:pseudodifferentialoperators} the pseudodifferential operator spaces considered here. Section~\ref{section:combine} discusses boundedness of the considered pseudodifferential operators on modulation spaces.

\subsection{ Modulation spaces}\label{section:modulationspaces}

To describe the full scope of operator sampling, we need to employ recent results in time--frequency analysis, in particular, we have to enter the realm of so-called modulation spaces. As Theorems~\ref{thm:main-simple} and \ref{thm:main-simple2} indicate, all results presented include the special case of Hilbert--Schmidt operators and the Hilbert space of square integrable functions as range space, and we advise readers without significant expertise in time--frequency analysis to focus on this case during a first reading.

Feichtinger introduced modulation spaces  in \cite{Fei81}. Modulation space theory was further developed  by Feichtinger and Gr\"ochenig as special case of their coorbit theory \cite{FG89}:  for $\rho$ being a square integrable unitary and irreducible representation of a locally compact group $G$ on a Hilbert space $H$ and $Y$ being a  Banach space  of functions on $G$, we consider, for appropriate $\varphi\in H$,   the so--called {\em voice transform} $V_\varphi : H\longrightarrow Y$ given by $V_\varphi f (x)=\langle f, \rho(x)\varphi \rangle$, $x\in G$. Given an appropriate Banach space Gelfand triple $X\subseteq H\subseteq X'$, the {\it coorbit space} $M_Y$ consists of those $f\in X'$ with $\|f\|_{M_Y}=\| V_\varphi f\|_Y<\infty$ \cite{FK98}.

The special case of modulation spaces is based on the Schr\"odinger representation of the reduced Weyl--Heisenberg group.  The corresponding voice transform simplifies to the short--time Fourier transform, that is, for any Schwartz class function $\varphi\neq 0$ we consider
\begin{eqnarray}
  V_\varphi f(\lambda)=\langle f , \pi(\lambda)\varphi\rangle=\mathcal F( f \, T_t \overline{\varphi})(\nu),\quad \lambda=(t,\nu) \in \R^{2d}, \notag
\end{eqnarray}
which is well defined for any $f\in \mathcal S'(\R^d)$ \cite{Gro01}. Note that throughout this paper, dual pairings $\langle\cdot,\cdot\rangle$ are linear in the first component and antilinear in the second. Moreover, any choice of $\varphi\in \mathcal S(\R^d)\setminus\{0\}$ can be used to define modulations spaces (with equivalent norms), but as is customary, we shall choose a normalized Gaussian, namely $\varphi(x)=\ga(x)=2^{\frac d 4}\,e^{-\pi\|x\|^2_2}$, $x\in\R^d$.

The role of the Banach space $Y$ in  coorbit space theory is attained in modulation space theory  by weighted mixed $L^p$ spaces which we shall describe now. For a measurable function $f$ on $\R^d$ and $p=(p_1,\ldots,p_d)$, $1\leq p_1,\ldots,p_d\leq \infty$, we define mixed $L^p(\R^d)$ spaces by finiteness of
$$
    \|f\|_{L^p}=\Big(   \int \Big( \ldots \Big(\int \Big(\int
            |f(x_1,\ldots,x_d)|^{p_1} dx_1 \Big)^{p_2/p_1} dx_2 \Big)^{p_3/p_2}\ldots
dx_{d-1}\Big)^{p_d/p_{d-1}}
            dx_d\Big)^{1/p_d},
$$
with the usual adjustments if some $p_k=\infty$ \cite{BP61}. The mixed $l^p(\Z^d)$ spaces are defined accordingly.

Note the sensitivity to the order of exponentiation and integration. For example, for $f(x,y)=1$ if $|x-y|\leq 1$ and $f(x,y)=0$ else, we have $\sup_{x} \int |f(x,y)| \, dy =2$ but  $ \int \sup_{x}|f(x,y)| \, dy =\infty$, that is, $ f\in L^{1,\infty}(\R^2)$ but $g\notin L^{\infty, 1}(\R^2)$ where $g(x,y)=f(y,x)$.

A locally integrable function $v:\R^{d}\longrightarrow \R^+_0$ with
$$
    v(x+y)\leq v(x)\, v(y),\quad x,y\in \R^{d},
$$
is called {\it submultiplicative weight}. For example, $w_s(x)=(1+\|x\|)^s$, $s\geq 0$, is a submultiplicative weight on $\R^{d}$. If  $v$ is a submultiplicative weight and the locally integrable function $w:\R^{d}\longrightarrow \R^+_0$ satisfies for some $C>0$
$$
    w(x+y)\leq C\, w(x)\, v(y),\quad x,y\in \R^{d},
$$
then $w$ is a {\em $v$-moderate weight function}. The class of $v$-moderate weight functions on $\R^d$ is denoted by $\mathcal M_v(\R^d)$. Note that for $s<0$, for example, $1{\otimes}w_{s}(x,\xi)=(1+\|\xi\|)^{s}$ is not submultiplicative, but $1{\otimes}w_{s}$ is $1{\otimes}w_{-s}$ -moderate. If $w$ is a $v$-moderate weight function with respect to some submultiplicative weight, then we simply say $w$ is {\em moderate}. Note that for any  moderate weight function on $\R^d$ exists $\gamma,C>0$ with  $\frac 1 C e^{-\gamma\|x\|_\infty} \leq  w(x)\leq C e^{\gamma\|x\|_\infty}$ (see Lemma 4.2 in \cite{Gro07b}). A moderate weight function $w$ on $\R^d$ is a {\it subexponential weight function} if there exists $\gamma, C>0$ and $0<\beta <1$ with
$$\tfrac 1 C e^{-\gamma\|x\|^\beta_\infty} \leq  w(x)\leq C e^{\gamma\|x\|^\beta_\infty}.$$

Weight functions on discrete groups such as $\Z^d$ are defined accordingly. See \cite{Gro07b} for a thorough discussion on the role of weight functions in time--frequency analysis.

Given a $v$-moderate weight function $w$, then the Banach space $L_w^p(\R^d)$  is defined through finiteness of the norm $\|f\|_{L^p_w}=\|wf\|_{L^p}$. The space $L_w^p(\R^d)$ is shift invariant, shift operators are bounded on $L_w^p(\R^d)$ but not  isometric if $w$ is not constant. Replacing $\R^d$ with $\Z^d$, or with a full rank lattice $\Lambda=A\Z^d$, $A\in \R^{d{\times}d}$ invertible, either equipped with the counting measure gives the definition of $l^p_w(\Z^d)$, respectively $l^p_w(\Lambda)$. If $w$ is a moderate weight on $\R^d$, then its restriction to $\Lambda$, which we denote by $\widetilde w$, is moderate as well.

\begin{definition} For $p=(p_1,\ldots,p_d)$ and
$q=(q_1,\ldots,q_d)$, $1\leq p_k,q_k\leq \infty$, and $w$ moderate on $\R^{2d}$, we define   modulation spaces by
\begin{eqnarray}
  M_w^{p,q}(\R^d)=\left\{f\in \mathcal S'(\R^d):\ V_{\ga}f \in L_w^{p,q}(\R^{2d})  \right
    \} \label{eqn:modulationSpace}
\end{eqnarray}
 \cite{Fei81,Gro01}. The modulation space $M_w^{p,q}(\R^d)$ is a shift invariant Banach space with norm $\|f\|_{M_w^{p,q}} = \| w V_{\ga}f  \|_{L^{p,q}}$. If $w\equiv 1$, then we write $M^{p,q}(\R^d)=M_w^{p,q}(\R^d)$. If $p_1=\ldots=p_d$ and $q_1=\ldots=q_d$  then we abbreviate $M_w^{p_1,q_1}(\R^d)= M_w^{(p_1,\ldots,p_d),(q_1,\ldots,q_d)}(\R^d)$.
\end{definition}
Below we shall use the fact that replacing $\ga$ with any other $\varphi\in \mathcal S(\R^d)\setminus \{0\}$ in \eqref{eqn:modulationSpace} defines the identical space with an equivalent norm \cite{Gro01}. Note that if $p_1\leq p_2$, \ $q_1\leq q_2$, and $w_1\geq c w_2$ for some $c>0$, then $M^{p_1,q_1}_{w_1}$ embeds continuously in $M^{p_2,q_2}_{w_2}$, and, consequently, if $w_1\asymp w_2$ then $M_{w_2}^{p,q}(\R^d)=M_{w_1}^{p,q}(\R^d)$ with equivalent norms.

The space $M^{1,1}(\R^d)$ is the Feichtinger algebra, often also denoted by $S_0(\R^d)$, and $M^{\infty,\infty}(\R^d)$ is its dual $S_0'(\R^d)$. In fact, in general we have $M_w^{p,q}(\R^d)'=M_{1/w}^{p',q'}(\R^d)$ for $1\leq p,q<\infty$ and $\frac 1 p +\frac 1 {p'}=1$, $\frac 1 q +\frac 1 {q'}=1$. Note that $M_{ 1{\otimes}w_{\rm s}}^{2,2}(\R^d)$ is also known as Bessel potential spaces, in particular $L^2(\R^d)=M^{2,2}(\R^d)$.

To illustrate the chosen order of exponentiation and integration in the definition of the modulation space $M_w^{p,q}(\R^d)$ for $d>1$ and $p\neq
q$, we state exemplary that $f\in M_{1{\otimes}w_s}^{(2,3),(4,5)}(\R^d)$  if and only if
$$
    \int\Big(\int\Big(\int\Big(\int \Big|(1+\sqrt{\nu_1^2+\nu_2^2})^s \, V_{\ga}f(t_1,t_2,\nu_1,\nu_2)\Big|^{2}
            \,dt_1\Big)^{\frac 3 2}\,dt_2\Big)^{\frac 4 3}\,d\nu_1\Big)^{\frac 5 4}\,d\nu_2\leq \infty.
$$
Clearly $f{\otimes} g\in M_{w_1{\otimes} w_2 }^{(p_1,p_2),(q_1,q_2)}(\R^{2d})$ if and only if $f \in
M_{w_1 }^{p_1,q_1}(\R^d)$ and $ g\in M_{w_2 }^{p_2,q_2}(\R^d)$. In this case,
$\|f{\otimes}
g\|_{M_{w_1{\otimes} w_2 }^{(p_1,p_2),(q_1,q_2)}}=\|f\|_{M_{w_1 }^{p_1,q_1}}\|g\|_{M_{ w_2 }^{p_2,q_2}}$.

For compactly supported and bandlimited functions, modulation spaces reduce to weighted mixed $L^p(\R^d)$ spaces. The following is a straight forward generalization of results in \cite{Fei89,Oko09}.

\begin{lemma}\label{lemma:compactbounded} Let $p=(p_1,\ldots,p_d)$ and
$q=(q_1,\ldots,q_d)$ with $1\leq p_k,q_k\leq \infty$, let $w=w_1\otimes w_2$ be a moderate weight function  on $\R^{2d}$, and suppose $M\subseteq \R^d$ compact. Then
\begin{enumerate}
    \item $\displaystyle \| f\|_{M_w^{p,q}}\asymp \|\widehat f\|_{L_{w_2}^{q}},\quad  f\in \mathcal S'(\R^d), \ \supp f\subseteq M $;
     \item $\displaystyle \| f\|_{M_w^{p,q}}\asymp \|f\|_{L_{w_1}^{p}},\quad  f\in \mathcal S'(\R^d), \ \supp \widehat f\subseteq M $.
\end{enumerate}
\end{lemma}

Modulation spaces allow for descriptions based on growth conditions of so-called Gabor coefficients \cite{Gro01}. These descriptions rely on the following terminology.

\begin{definition}\label{definition:p-frames}
  Let $X$ be a Banach space,
  $1\leq p_1,\ldots,p_d\leq \infty$, and let $w$ be moderate on the full rank lattice $\Lambda$.

  \begin{enumerate}

    \item  $\{g_\lambda\}_{\lambda \in \Lambda}\subseteq X'$ is called $l^p_{ w}$--frame for  $X$,
  if the analysis operator
  $\displaystyle C_{\{g_\lambda\}}: X\longrightarrow l_{w}^p(\Lambda),
    \quad f\mapsto
        \{\langle f, g_\lambda \rangle\}_{\lambda\in\Lambda}\,,
  $
 is well defined and
  \begin{eqnarray}
   \|f\|_X \asymp \| \{\langle f, g_\lambda \rangle\} \|_{l^p_{ w}},\quad f\in X\,.
      \label{equation:NormEquivalence}
    \end{eqnarray}

    \item $\{g_\lambda\}_{\lambda \in \Lambda}\subseteq X$ is  called $l^p_{w}$--Riesz basis in $X$,
  if the synthesis operator $
    D_{\{g_\lambda\}}: l^p_{ w}(\Lambda)\longrightarrow X,\quad \{c_\lambda\}_{\lambda\in\Lambda} \mapsto
            \sum_\lambda c_\lambda g_\lambda
$,  is well defined and
  \begin{eqnarray}
      \|\{c_\lambda\}\|_{l^p_{w}}\asymp \|\sum_\lambda c_\lambda g_\lambda  \|_{X},\quad \{c_\lambda\} \in l^p_{ w}(\Lambda).
      \label{equation:BanachRiesz}
    \end{eqnarray}
  \end{enumerate}
\end{definition}

In the classical Hilbert space setting $X=X'=H$ and $l^p_w(\Z^{2d})=l^2(\Z^{2d})$, the above entails the definition of  Hilbert space frames and Riesz basis sequences. In the Hilbert space theory, condition \eqref{equation:NormEquivalence} implies that $\displaystyle C_{\{g_\lambda\}}$ has a bounded left inverse, but in the general Banach space setting, \eqref{equation:NormEquivalence} alone does not guarantee the existence of such a left inverse. Therefore, the condition of a bounded left inverse $\displaystyle C_{\mathcal F}$ is frequently included in the definition of frames for Banach spaces \cite{Chr03,Gro91,FZ98}.%CH97?

Note that for any $1\leq p \leq \infty$, $w$ moderate, $l^p_w$--Riesz bases form unconditional bases for their closed linear span. This follows directly from \eqref{equation:BanachRiesz} and Definition 12.3.1 and Lemma 12.3.6 in \cite{Gro01}.

For $g\in \mathcal S(\R^d)$ and $\Lambda$ being a full rank lattice in $\R^{2d}$, we set $(g,\Lambda)=\{\pi(\lambda)g\}_{\lambda\in \Lambda}$.  For $w$ moderate on $\R^{2d}$, set $\widetilde w_\lambda=w(\lambda)$. Results as Theorem~\ref{thm:GaborEssentails} are important tools in modulation space theory, see, for example, Theorem 20 in \cite{Gro04} or Theorem 6.11 in \cite{Gro07b}.

\begin{theorem}\label{thm:GaborEssentails}
Let  $1\leq p,q\leq \infty$ and  let $w$ be moderate on $\R^{2d}$. Let $\Lambda$ be a full rank lattice in $\R^{2d}$ and $g\in \mathcal S (\R^d)$.

\begin{enumerate}
 \item  If $(g,\Lambda)$ is a frame for $L^2(\R^d)$, then $(g,\Lambda)$ is an $l^{p,q}_{\widetilde w}$--frame for $M^{p,q}_w(\R^d)$.
     \item   If $(g,\Lambda)$ is a Riesz basis in $L^2(\R^d)$, then $(g,\Lambda)$ is an $l^{p,q}_{\widetilde w}$--Riesz basis in $M^{p,q}_w(\R^d)$.
\end{enumerate}
\end{theorem}
\begin{proof}
 {\it 1.} Let $g\in \mathcal S (\R^d)$ with $(g,\Lambda)$ is a frame for $L^2(\R^d)$. Let $\widetilde g$ generate the canonical dual frame of $(\widetilde g, \Lambda)$ of $(g,\Lambda)$ \cite{Gro01}. We have $\widetilde g \in\calS(\R^d)$ \cite{Jan95} and conclude that both, $C_{(g,\Lambda)}:M^{p,q}_w(\R^d) \longrightarrow l^{p,q}_{\widetilde w}(\Lambda)$ and $D_{(\widetilde g,\Lambda)}: l^{p,q}_{\widetilde w}(\Lambda) \longrightarrow M^{p,q}_w(\R^d)$ are bounded operators. As $D_{(\widetilde g,\Lambda)} {\circ} C_{(g,\Lambda)}$ is the identity on $L^2(\R^d)$, we can use a density argument to obtain $D_{(\widetilde g,\Lambda)} {\circ} C_{(g,\Lambda)}$ is the identity on  $M^{p,q}_w(\R^d)$. We conclude that $C_{(g,\Lambda)}$ is bounded below.

The proof of {\it 2.} follows similarly.
\end{proof}

\subsection{Time--frequency analysis of pseudodifferential operators}\label{section:pseudodifferentialoperators}

The framework of Hilbert--Schmidt operators suffices to develop the basics of our sampling theory for operators. But important operators such as convolution operators, multiplication operators, and even the identity  are not compact and thereby fall outside the realm of Hilbert--Schmidt operators. Rather than focusing on operators with kernel in $L^2(\R^{2d})$, we shall consider kernels and symbols in modulation spaces.

To formulate a widely applicable sampling theory for operators, we use the general correspondence of operators to distributional kernels given by the Schwartz kernel theorem (see, for example, \cite{Hor07}).

\begin{theorem}\label{thmschwartz}
For any linear and continuous operator $H:\mathcal S(\R^d) \longrightarrow \mathcal S'(\R^d)$
exists a unique $\kappa_H\in \mathcal S'(\R^{2d})$ with $\langle Hf, g\rangle = \langle\kappa_H , \overline{f} \otimes g \rangle $, $f,g\in \mathcal S(\R^d)$.
\end{theorem}

Alternatively to $\kappa_H$, we can consider the so--called time--varying impulse response $h_H\in S'(\R^{2d})$ of $H:\mathcal S(\R^d) \longrightarrow \mathcal S'(\R^d)$ which is formally given by
\begin{eqnarray*}
   h_H(x,t)=\kappa_H(x,x-t),\quad Hf(x)=\int h_H(x,t) f(x-t)\,dt.
\end{eqnarray*}
The Kohn--Nirenberg symbol $\sigma_H$ of an operator $H:\mathcal S(\R^d) \longrightarrow \mathcal S'(\R^d)$  is densely defined by $\sigma_H=\mathcal F_{t\to \xi} h_H$, that is,
\begin{eqnarray}
    \sigma_H(x,\xi)
            =\int \kappa_H(x,x-t)\,e^{-2\pi i t\xi}\, dt, \quad  Hf(x)=\int \sigma_H(x,\xi)\, \widehat f (\xi) \, e^{2\pi i x \xi }\,d\xi \notag
          %  \label{eqnarray:KohnNirenberg}
\end{eqnarray}
\cite{Fol89,KN65}. Note that the  $n$-th order linear differential operator $D:f\mapsto \sum_{n=0}^N a_n(x) f^{(n)}(x)$ has Kohn--Nirenberg symbol $\sigma_D(x,\xi)=\sum_{n=0}^N a_n(x)  (2\pi i \xi)^n$ which is polynomial in $\xi$. Pseudodifferential operator classes, for example, those considered by H\"ormander,  have symbols $\sigma_H$ which are not necessarily polynomial in $\xi$ but which satisfy corresponding polynomial growth conditions \cite{Hor07}.

Additionally, in time--frequency analysis and in communications engineering, the spreading function $\eta_H$
 is commonly used to describe $H$:
\begin{eqnarray}\eta_H=\mathcal F^s \sigma_H, \quad
Hf(x) =  \iint\eta_H (t,\nu)\,M_\nu  T_t f(x)
 \, dt\,d\nu\, .  \label{equation:spreadingrepresentationBochner}
\end{eqnarray}
Equation \eqref{equation:spreadingrepresentationBochner} can be validated weakly by first integrating with respect to $x$ in
\begin{eqnarray}\notag % \label{eqn:spreadingandstft}
    \langle H f ,\varphi \rangle
        &=& \int\!\int\!\!\!\!\!\int \eta_H(t,\nu)\pi(t,\nu)f(x)
            \overline{\varphi(x)} \, dt\,d\nu\ dx % \notag \\
       % &=& \int\!\!\!\!\!\int \eta_H(t,\nu)
       %  \overline{\int \varphi(x) \overline{\pi(t,\nu)f(x)}\, dx}  \
        %    dt\,d\nu
        =  \langle \eta_H,V_f \varphi \rangle, \quad f,\varphi\in {\mathcal S}(\R^d), %\notag
\end{eqnarray}
where $V_f \varphi (t,\nu)=\langle \varphi, \pi(t,\nu) f \rangle$ is the short time Fourier transform  defined above. Equation \eqref{equation:spreadingrepresentationBochner} illustrates that support restrictions on $\eta_H$ reflect limitations on the maximal time and frequency shifts which the input signals undergo: $Hf$ is a continuous superposition of time--frequency shifted versions of $f$ with weight function $\eta_H$ \cite{GP08,KP06,PW06b}.
Moreover, as $h(x,t)=\int \eta(t,\nu)e^{2\pi i \nu x}\,d\nu$,  the
condition $\supp \eta_H(t,\cdot)\subseteq [-\frac b 2,\frac b 2]$, $t\in\R$, excludes high frequencies and therefore rapid change of the time--varying impulse response $h(x,t)$ as a function of $x$. In the time--invariant case, $\kappa(x,x-t)=h(x,t)=h(t)$ is, in fact, independent of $x$. These observations  illuminate the role of support constraints on spreading functions in the analysis of {\it slowly time--varying} communications channels \cite{Bel63,Zad52}. Additional aspects on the use of pseudodifferential operator calculus in communications can be found in \cite{Stro06}.

\subsection{Boundedness of pseudodifferential operators on modulation spaces}\label{section:combine}

Theorem~\ref{thm:OPWextend} in Section~\ref{section:MainResults1} provides the upper bound in \eqref{eqn:identifiable} for Theorems~\ref{thm:main-simple2}, \ref{thm:main-full}, \ref{thm:main-fullb}, \ref{thm:symplectic}, and \ref{thm:main-full2}. It follows from Theorem~\ref{theorem:boundedoperators} which generalizes Theorem~4.2 in \cite{Tof07} as well as results in \cite{CG03,Cza03,GH99,GH04,Tac94a,Tof04} where, generally, the case $p_3=q_3$ and $p_4=q_4$ in the notation  below was considered. Recall that $p'$ denotes the conjugate exponent of  $1\leq p\leq \infty$, that is $\frac 1 {p} + \frac 1 {p'} =1$.

\begin{theorem}\label{theorem:boundedoperators}
Assume $1\leq p_1,p_2,p_3, p_4,q_1,q_2,q_3,q_4\leq \infty$ with $p_4\leq q_3,q_4$,
\begin{eqnarray}
  1+ \frac{1}{p_2}\leq \frac{1}{p_1} + \frac{1}{p_3} + \frac{1}{p_4} \quad \text{and} \quad  1+ \frac{1}{q_2} \leq \frac{1}{q_1} + \frac{1}{q_3} + \frac{1}{q_4}\,.
    \label{equation:ArithmeticOperators}
  \end{eqnarray}
Let the moderate weight functions $\weight,\weight_1,\weight_2$ satisfy
\begin{eqnarray}
  \weight(x,\xi,\nu,t)\geq c\ \frac{\weight_2(x,\nu+\xi)}{\weight_1(t-x,\xi)}  \notag %\label{eqn:weightinequality}
\end{eqnarray}
with $c>0$. Then, for some $C>0$,
\begin{eqnarray} \| L_\sigma f \|_{M_{\weight_2}^{p_2,q_2}}\leq C\ \| \sigma \|_{M_\weight^{(p_3,q_3),(q_4,p_4)}} \ \|f\|_{M_{\weight_1}^{p_1,q_1}},\quad f\in M_{\weight_1}^{p_1,q_1}(\R^d),\ \sigma \in M_\weight^{(p_3,q_3),(q_4,p_4)}(\R^{2d}) \label{eqn:norminequalitysigma} ,
\end{eqnarray}
consequently, $L_\sigma:\,  M_{\weight_1}^{p_1,q_1}(\R^d)\longrightarrow  M_{\weight_2}^{p_2,q_2}(\R^d)$ is bounded for $\sigma \in M_\weight^{(p_3,q_3),(q_4,p_4)}(\R^{2d})$ and $L_\sigma$ denoting the operator corresponding to the Kohn-Nirenberg symbol $\sigma$.
\end{theorem}

Theorem~\ref{theorem:boundedoperators} is a consequence of the following lemma.

\begin{lemma} \label{lem:modulationspacemembership} Assume $1\leq p_1,p_2,p_3, p_4,q_1,q_2,q_3,q_4\leq \infty$ with  $p_3\leq p_1, p_2,p_4$, \ $q_3\leq q_1,q_2, q_4$,
 \begin{eqnarray}
  \frac{1}{p_1} +\frac{1}{p_2}=\frac{1}{p_3} +\frac{1}{p_4} \quad \text{and} \quad \frac{1}{q_1} +\frac{1}{q_2}=\frac{1}{q_3} +\frac{1}{q_4}.
    \label{equation:Arithmetic}
  \end{eqnarray}
  Let the moderate weight functions $w, w_1, w_2$ satisfy
\begin{eqnarray}
  w(x,t,\nu,\xi)\leq w_1(t-x,\xi)w_2(x,\nu+\xi)\,. \label{eqn:weightinequality}
\end{eqnarray}
Then for $\Ga(x,\xi)=\ga(x)\ga(x-t)$, we have
\begin{eqnarray}
  \Big( \int \Big( \int \Big( \int \Big( \int &&\hspace{-.9cm}
    \big| V_\Ga
    \overline{f}{\otimes} g {\circ} \left( \begin{smallmatrix} 1 & -1 \\ 1 & \ 0 \end{smallmatrix} \right)(x,t,\nu,\xi)\, w(x,t,\nu,\xi)\big|^{p_3}\, dx \Big)^{\frac {p_4}{p_3}}\, dt \Big)^{\frac {q_3}{p_4}}\,d\xi \Big)^{\frac {q_4}{q_3}}\, d\nu\Big)^{\frac 1 {q_4}}\notag \\ &\leq& \| f\|_{M_{\weight_1}^{p_1,q_1}}\ \|g\|_{M_{\weight_2}^{p_2,q_2}}, \quad f\in M_{\weight_1}^{p_1,q_1}(\R^d), g\in M_{\weight_2}^{p_2,q_2}(\R^d) \, ,\label{eqn:norminequality}
\end{eqnarray}
  where $ \overline{f}{\otimes} g {\circ} \left( \begin{smallmatrix} 1 & -1 \\ 1 & \ 0 \end{smallmatrix} \right)(x,t)=\overline{f}(x-t)g(x)$.

\end{lemma}

\begin{proof}
For $g,f\in \mathcal S(\R^d)$, we compute
\begin{eqnarray}
V_\Ga\overline{f}{\otimes} g {\circ} \left( \begin{smallmatrix} 1 & {-}1 \\ 1 & \ 0 \end{smallmatrix} \right) (x,t,\nu,\xi)
    &=& \iint g (x')\, \overline{f} (x'{-}t')  e^{{-} 2\pi i(x'\nu{+}t'\xi)} \ga (x'{-}x)\ga (x'{-}x {-} (t'{-}t))\, dt'\, dx' \notag \\
    &=& \int g(x') e^{{-} 2\pi i x'\nu}  \ga (x'{-}x) \int \overline{f}(s) e^{{-}2\pi i (x'{-}s)\xi} \ga (s{-}(x{-}t))\, ds\, dx'\notag \\
    &=& \int g(x') e^{{-} 2\pi i x'(\nu{+}\xi)}  \ga (x'{-}x) \, dx'\ \overline{  \int f(s) e^{{-}2\pi i s\xi} \ga (s{-}(x{-}t))\, ds}\notag \\
    &=& V_{ \ga} g (x, \nu{+}\xi)\ \overline{V_\ga f (x{-}t, \xi) }\,.
\end{eqnarray}
Assume $1\leq p_1,p_2,p_3,p_4, q_1,q_2,q_3,q_4 <\infty$. For $\weight\equiv 1$ and $\weight_1=\weight_2\equiv 1$ we have
\begin{eqnarray}
  \Big( \int \Big( \int \Big( \int \Big( \int &&\hspace{-.9cm}
    \big| V_\Ga
    \overline{f}{\otimes} g {\circ} \left( \begin{smallmatrix} 1 & -1 \\ 1 & \ 0 \end{smallmatrix} \right)(x,t,\nu,\xi)\big|^{p_3}\, dx \Big)^{\frac {p_4}{p_3}}\, dt \Big)^{\frac {q_3}{p_4}}\,d\xi \Big)^{\frac {q_4}{q_3}}\, d\nu\Big)^{\frac 1 {q_4}}\notag \\
  &=&
  \Big(\int \Big(\int \Big(\int \Big(\int |V_\ga f(x{-}t,\xi)\ V_\ga g (x, \nu+\xi)|^{p_3} \, dx\Big)^{\frac{p_4}{p_3}}\,dt\Big)^{\frac{q_3}{p_4}}\,d\xi\Big)^{\frac{q_4}{q_3}}d\nu\Big)^{\frac{1}{q_4}}
\label{eqn:firstequation}
\\
%&=&
%  \Big(\int \Big(\int \Big(\int F_\nu(\xi,t)\,d\xi\Big)^{\frac{p_4}{q_3}}\,dt\Big)^{\frac{q_3}{p_4} \frac{p_4}{q_3} \frac{q_4}{p_4}}d\nu\Big)^{\frac{1}{q_4}} \notag\\
%&=&
%  \Big(\int  \Big\| \int F_\nu(\xi,\cdot)\,d\xi\Big\|_{L^{\frac{p_4}{q_3}}}\Big)^{\frac{q_4}{q_3} }d\nu\Big)^{\frac{1}{q_4}} \notag\\
%&\leq&
%  \Big(\int \Big( \int \| F_\nu(\xi,\cdot)\|_{L^{\frac{p_4}{q_3}}}\,d\xi\Big)^{\frac{q_4}{q_3}}\,d\nu\Big)^{\frac{1}{q_4}} \label{equation:Minkowski}\\
%&=&
%  \Big(\int \Big( \int \Big(\int  F_\nu(\xi,t)^{\frac{p_4}{q_3}}\,dt\Big)^{\frac{q_3}{p_4}} \,d\xi\Big)^{\frac{q_4}{q_3}}\,d\nu\Big)^{\frac{1}{q_4}} \notag \\
%&=&
%  \Big(\int \Big( \int \Big(\int  \Big(\int |V_\ga f(x{-}t,\xi)\ V_\ga g (x, \nu+\xi)|^{p_3} \, dx\Big)^{\frac{q_3}{p_3}\frac{p_4}{q_3}}\,dt\Big)^{\frac{q_3}{p_4}} \,d\xi\Big)^{\frac{q_4}{q_3}}\,d\nu\Big)^{\frac{1}{q_4}} \notag \\
%&=&
%  \Big(\int \Big( \int \Big(\int  \Big(\int |V_\ga f(x{-}t,\xi)\ V_\ga g (x, \nu+\xi)|^{p_3} \, dx\Big)^{\frac{p_4}{p_3}}\,dt\Big)^{\frac{q_3}{p_4}} \,d\xi\Big)^{\frac{q_4}{q_3}}\,d\nu\Big)^{\frac{1}{q_4}} \notag \\
&\leq&
  \Big(\int \Big( \int \|V_\ga f(\cdot,\xi)^{p_3}\|_{L^{r_1}} \|V_\ga g (\cdot, \nu+\xi)^{p_3}\|_{L^{s_1}}\Big)^{\frac{p_4}{p_3}\frac{q_3}{p_4}} \,d\xi\Big)^{\frac{q_4}{q_3}}\,d\nu\Big)^{\frac{1}{q_4}} \label{equation:YoungFirst} \\
&=&
  \Big(\int \Big( \int \|V_\ga f(\cdot,\xi)^{p_3}\|_{L^{r_1}} \|V_\ga g (\cdot, \nu+\xi)^{p_3}\|_{L^{s_1}}\Big)^{\frac{q_3}{p_3}} \,d\xi\Big)^{\frac{q_4}{q_3}}\,d\nu\Big)^{\frac{q_3}{q_4}\frac{q_4}{q_3}\frac{1}{q_4}}\notag \\
&\leq&
  \Big\| \ \|V_\ga f^{p_3}\|_{L^{r_1}}^{\frac{q_3}{p_3}} \Big\|_{L^{r_2}}^{\frac{1}{q_3}}
  \ \
  \Big\| \ \|V_\ga g ^{p_3}\|_{L^{s_1}}^{\frac{q_3}{p_3}} \Big\|_{L^{s_2}}^{\frac{1}{q_3}} \label{equation:YoungSecond} \\
&=&
  \Big(\int\Big(\int |V_\ga f| ^{p_3r_1}\,dx\Big)^{\frac{r_2}{r_1}\frac{q_3}{p_3}} \, d\xi\Big)^{\frac{1}{r_2q_3}}
  \ \
 \Big(\int\Big(\int |V_\ga g ^{p_3s_1}\,dx\Big)^{\frac{s_2}{s_1}\frac{q_3}{p_3}} \, d\xi\Big)^{\frac{1}{s_2q_3}} \notag % \label{eqn:independentPHI}
\\
&\equiv&
  \|f\|_{M^{p_3r_1,q_3r_2}}
  \ \
 \|g\|_{M^{p_3s_1,q_3s_2}} \,. \notag
\end{eqnarray}
To apply Young's inequality to obtain \eqref{equation:YoungFirst}, we assume $p_4\geq p_3$ and choose $r_1,s_1\geq 1$ with
\begin{eqnarray}
  \frac{1}{r_1}+\frac{1}{s_1}=1+\frac{p_3}{p_4}. \label{equation:YoungFirstArtihmetic}
\end{eqnarray} Similarly,  to obtain \eqref{equation:YoungSecond}, we use  $q_4\geq q_3$ and choose $r_2,s_2\geq 1$ with
\begin{eqnarray}
  \frac{1}{r_2}+\frac{1}{s_2}=1+\frac{q_3}{q_4}. \label{equation:YoungSecondArtihmetic}
\end{eqnarray}
%To obtain \eqref{eqn:independentPHI} we used that replacing $\ga$ by any $\phi\in \mathcal S(\R^d)$ in \eqref{eqn:modulationSpace} leads to a norm equivalent to $\|\cdot \|_{M^{p,q}_w}$.

To conclude our proof of the unweighted case, we set $p_1=p_3r_1$, $q_1=q_3r_2$, $p_2=p_3s_1$, and $q_2=q_3s_2$. As all factors must be greater than or equal to one, we require $p_1,p_2\geq p_3$ and $q_1,q_2\geq q_3$. Moreover, \eqref{equation:YoungFirstArtihmetic} and \eqref{equation:YoungSecondArtihmetic} need to be satisfied, this holds if and only if \eqref{equation:Arithmetic} holds. The case that for some $k$,  $p_k=\infty$ or $q_k=\infty$ follows from making the usual adjustments.

The weighted case follows by simply replacing $V_\Ga G$ with $\weight \,V_\Ga G$ in equations  \eqref{eqn:firstequation} till \eqref{equation:YoungFirst}, and then replacing $V_{\ga} f$ and $V_\ga g$ by $\weight_1\, V_{ \ga} f$ and $\weight_2 V_\ga g$. This is justified by \eqref{eqn:weightinequality}.
\end{proof}

{\it Proof of Theorem~\ref{theorem:boundedoperators}.}
Let $f,g\in \mathcal S(\R^d)$ and $H$ with $\sigma_H\in M^{(p_3,q_3),(q_4,p_4) }(\R^{2d})$.  Then
\begin{eqnarray}
\big|\langle H f, g \rangle\big|&=&\big|\int \int h(x,t)  f (x-t)\,dt \  \overline g (x) \, dx\big|=\big|\langle h_H, \overline{f}{\otimes} g {\circ} \left( \begin{smallmatrix} 1 & -1 \\ 1 & \ 0 \end{smallmatrix} \right) \rangle\big|
  \notag %\label{eqn:KNsymbol}
  \\
  &=&\big|\langle \sigma_H, \mathcal F_{t\rightarrow \xi} \overline{f}{\otimes} g {\circ} \left( \begin{smallmatrix} 1 & -1 \\ 1 & \ 0 \end{smallmatrix} \right) \rangle\big| \notag \\
  &\leq&\big\| \sigma_H\|_{M_w^{(p_3,q_3),(q_4,p_4)}} \ \big\| \mathcal F_{t\rightarrow \xi} \overline{f}{\otimes} g {\circ} \left( \begin{smallmatrix} 1 & -1 \\ 1 & \ 0 \end{smallmatrix} \right) \big\|_{M_{1/ w}^{(p_3',q_3'),(q_4',p_4')}} \label{eqn:Hoelder} \,,
\end{eqnarray} where we applied H\"older's inequality for weighted mixed $L^p$-spaces to obtain \eqref{eqn:Hoelder} \cite{Gro01}.
To obtain \eqref{eqn:norminequalitysigma}, it suffices to show $ \mathcal F_{t\rightarrow \xi}\overline{f}{\otimes} g {\circ} \left( \begin{smallmatrix} 1 & -1 \\ 1 & \ 0 \end{smallmatrix}\right)  \in M_{1/ w}^{(p_3',q_3'),(q_4',p_4')}$  for $f\in M_{w_1}^{p_1,q_1}$ and $g\in M_{1/ w_2}^{p_2',q_2'}$. Note that replacing $\ga$   by any other test function  in \eqref{eqn:modulationSpace} leads to a norm equivalent to $\|\cdot \|_{M^{p,q}_w}$, and we choose to show that for $\Psi=\mathcal F_{t\rightarrow \xi} \Ga$,  $\Ga(x,\xi)=\ga(x)\ga(x-t)$, we have that
\begin{eqnarray}
  \big\| \tfrac 1 w V_\Psi \mathcal F_{t\rightarrow \xi}\overline{f}{\otimes} g {\circ} \left( \begin{smallmatrix} 1 & -1 \\ 1 & \ 0 \end{smallmatrix}\right)\|_{L^{p_3',q_3',q_4',p_4'}} \label{eqn:righthandside}.
\end{eqnarray}
is bounded by the left hand side in \eqref{eqn:norminequality} for $f\in M_{w_1}^{p_1,q_1}$ and $g\in M_{1/ w_2}^{p_2',q_2'}$.  Note that  as
$$\big|\tfrac 1 w V_\Psi \mathcal F_{t\rightarrow \xi}\overline{f}{\otimes} g {\circ} \left( \begin{smallmatrix} 1 & -1 \\ 1 & \ 0 \end{smallmatrix}\right)\big|(x,\xi,\nu,t)=\big|\tfrac 1 w V_\Ga \overline{f}{\otimes} g {\circ} \left( \begin{smallmatrix} 1 & -1 \\ 1 & \ 0 \end{smallmatrix}\right)\big|(x,t,\nu,\xi),\quad x,\xi,t,\nu\in\R^d,$$
the boundedness follows from an adjustment of the order of exponentiation and integration in \eqref{eqn:norminequality}.
Minkowski's integral inequality, namely,
$$
   \Big\| \int |w(x,\cdot )\,F(x,\cdot)|^p \, dx\Big\|_{L^{\frac q p}} \leq  \int \Big(  \int |w(x,y)\,F(x,y)|^q \,dy \Big)^{\frac p q} dx
$$
implies that
 $\|w(x,y)f(x,y)\|_{L^{p,q}} \leq \|w(y,x)f(y,x)\|_{L^{qp}}$ if $p\leq q$. Hence, if $q_4'\leq p_4'$ and $q_3'\leq p_4'$, then we can move the $t$-integral in between the $x$-integral and the $\xi$-integral and obtain that \eqref{eqn:righthandside} is bounded by the left hand side of \eqref{eqn:norminequality}.

We now prepare to apply Lemma~\ref{lem:modulationspacemembership}.  Observe that if we assume
\begin{eqnarray}
  p_4', p_1,p_2'\geq p_3',\quad q_4', q_1,q_2'\geq q_3',\quad p_4'\geq q_3',q_4'\, , \label{eqn:exactinequals}
\end{eqnarray}
and
 \begin{eqnarray}
  \frac{1}{p_1} +\frac{1}{p_2'}=\frac{1}{p_3'} +\frac{1}{p_4'} \quad \text{and} \quad \frac{1}{q_1} +\frac{1}{q_2'}=\frac{1}{q_3'} +\frac{1}{q_4'},
  \notag
  \end{eqnarray} that is,
  $p_4, p_1',p_2\leq p_3$ and  $q_4, q_1',q_2\leq q_3$ and  $p_4\leq q_3,q_4$, and
  \begin{eqnarray}
  \frac{1}{p_1} +1- \frac{1}{p_2}=1- \frac{1}{p_3} + 1- \frac{1}{p_4} \quad \text{and} \quad \frac{1}{q_1} +1- \frac{1}{q_2}=1- \frac{1}{q_3} +1- \frac{1}{q_4},
  \notag
  \end{eqnarray}
  the latter being
  \begin{eqnarray}
  \frac{1}{p_1} - \frac{1}{p_2}=1- \frac{1}{p_3} - \frac{1}{p_4} \quad \text{and} \quad \frac{1}{q_1} - \frac{1}{q_2}=1- \frac{1}{q_3} + \frac{1}{q_4},
  \notag
  \end{eqnarray}
  and
  \begin{eqnarray}
  1+\frac{1}{p_2} = \frac{1}{p_1}+ \frac{1}{p_3} +\frac{1}{p_4} \quad \text{and} \quad 1+\frac{1}{q_2} = \frac{1}{q_1}+ \frac{1}{q_3} + \frac{1}{q_4}.
  \notag
  \end{eqnarray}
  Hence, we obtain \eqref{eqn:norminequalitysigma} if \eqref{eqn:exactinequals} is satisfied.
  Note that for $\widetilde p\leq p$ and $\widetilde q\leq q$ we have $M_w^{\widetilde p, \widetilde q}$ embeds continuously in $M_w^{ p, q}$ (see, for example, Theorem 12.2.2 in \cite{Gro01}).  Hence, \eqref{eqn:norminequalitysigma} remains true if we decrease $p_1,p_3,p_4$ and $q_1,q_3,q_4$, and/or increase $p_2$ and $q_2$. We conclude that  \eqref{eqn:norminequalitysigma} holds if  \eqref{equation:ArithmeticOperators}  and $p_4\leq q_3,q_4$ are satisfied.

\hfill $\square$

\begin{remark}
  \rm  Note that for Hilbert--Schmidt operators, we have
\begin{eqnarray}
     \|H\|_{HS}=\|\kappa_H\|_{L^2}=\|h_H\|_{L^2}= \|\sigma_H\|_{L^2}= \|\eta_H\|_{L^2}, \label{eqn:HSequality}
\end{eqnarray}
a fact which is helpful to obtain norm inequalities of the form \eqref{eqn:identifiable}. But when considering modulation space norms for operator symbols, the chain of equalities \eqref{eqn:HSequality}   fails to hold. For example, we  have
$$
    |\langle h_H, \pi(x,t,\nu,\xi)\ga \rangle|
        =   |\langle \sigma_H, \pi(x,\xi,\nu,t)\ga \rangle|
        =   |\langle \eta_H, \pi(t,\nu,\xi,x)  \ga \rangle|,
$$
but due to the implicitly given order of exponentiation and integration,
$$
     \|h_H\|_{M^{(p_1,p_2),(q_1,q_2)}}\not\asymp \|\sigma_H\|_{M^{(p_1,q_2),(q_1,p_2)}}\not\asymp \|\eta_H\|_{M^{(p_2,q_1),(q_2,p_1)}},\quad H:\mathcal S(\R^d)\longrightarrow \mathcal S'(\R^d).
$$
Consequently, when defining a modulation space type norm on sets of pseudodifferential operators, one can base it on either $h_H$, $\sigma_H$, or $\eta_H$, each choice leading to different operator spaces and norms. Lemma~\ref{lem:modulationspacemembership} gives a hint that it may be advantageous to define operator modulation spaces\\ $OM^{p_1,p_2,q_1,q_2}(L^2(\R^d),L^2(\R^d))$ on $L^2(\R^d)$ through finiteness of the norm
\begin{eqnarray}
 \|H\|_{OM_w^{p_1,p_2,q_1,q_2}}= \Big( \int \Big( \int \Big( \int \Big( \int
    \big| V^s_\ga \sigma_H (x,t,\xi,\nu)\, w(x,t,\xi,\nu)\big|^{p_1}\, dx \Big)^{\frac {p_2}{p_1}}\, dt \Big)^{\frac {q_1}{p_2}}\,d\xi \Big)^{\frac {q_2}{q_1}}\, d\nu\Big)^{\frac 1 {q_2}}\, ,\notag
\end{eqnarray}
where the {\it symplectic short-time Fourier transform} $V^s$ with respect to the window function $\ga\in \mathcal S(\R^{2d})$ is given by
$$
    V^s_\ga F (x,t,\xi,\nu)=\mathcal F^s \big(F\cdot \overline{T_{x,\xi} \ga}\big)(t,\nu), \quad F\in\mathcal S'(\R^{2d}).
$$
This choice of order of exponentiation and integration respects arranging the {\it time} variables ahead of the {\it frequency} variables, while listing first the {\it absolute time} variable $x$ and then the {\it time-shift} variable $t$, respectively, we first list the {\it absolute frequency} variable $\xi$ and then the {\it frequency-shift} variable $\nu$. More importantly,  with this choice, we have
$$ \| H f \|_{M^{p_2,q_2}}\leq C \|H\|_{OM^{p_3,p_4,q_3,q_4}} \ \|f\|_{M^{p_1,q_1}},\quad H \in OM^{p_1,p_2,q_1,q_2}(M^{p_1,q_1} (\R^d), {M^{p_2,q_2}}(\R^d) )$$
for  all $1\leq p_1,p_2,p_3, p_4,q_1,q_2,q_3,q_4\leq \infty$  satisfying \eqref{equation:ArithmeticOperators}.

For simplicity of terminology, we avoid the use of operator modulation spaces and  symplectic short-time Fourier transforms in the following. Lemma~\ref{lemma:compactbounded} implies that this does not lead to a loss of generality in case of the here considered operator Paley--Wiener spaces.
\end{remark}

\section{Sampling and reconstruction in operator Paley--Wiener spaces}\label{section:MainResults1}
We introduce  {\em operator Paley--Wiener spaces}.

 \begin{definition}
 For $1\leq p,q\leq\infty$ and a moderate weight $w$ on $\R^{2d}$,  operator Paley--Wiener spaces are given by
\begin{eqnarray}
  OPW^{p,q}_{w}(M)=\{H:\mathcal S(\R^d)\longrightarrow \mathcal S'(\R^d): \ \supp \mathcal F^s \sigma_H
\subseteq
  M\text{ and } \sigma_H \in L_{w}^{p,q}(\R^{2d})\}. \notag
\end{eqnarray}
$OPW^{p,q}_w(M) $ is a Banach space with norm $\|H\|_{OPW^{p,q}_{\weight}}=\|\sigma_H\|_{L_w^{p,q}}$. If $w\equiv 1$ and $p=q=2$ then we simply write $OPW(M)=\big\{  H\in HS(L^2(\R^d)):\quad \supp \mathcal F_s \sigma_H\subseteq M  \big\}$.
 \end{definition}

Note that,  as illustrated in Corollary~\ref{thmimproved-sampling} and Example~\ref{exa:convolution} below, it is appropriate to choose $OPW^{p,\infty}_w(M)$, respectively  $OPW^{\infty,q}_w(M)$, when considering multiplication respectively convolution operators. Moreover, observe that $OPW_w^{\infty,\infty}(M)$ consists of all operators in the weighted Sj\"ostrand class  with Kohn--Nirenberg symbol bandlimited to $M$ \cite{Gro06,Sjo94,Sjo95,Stro06}.

\begin{remark}\rm
In \cite{Hor07}, H\"ormander considers pseudodifferential operators with Kohn--Nirenberg symbol in
$$
    S^m_{\rho,\delta}=\big\{ \sigma \in C^\infty(\R^2d):\
            \big|\partial^\alpha_\xi \partial^\beta_x \sigma (x,\xi)\big|\leq C_{\alpha,\beta}(1+\|\xi\|_2)^{m-\rho (\alpha_1+\ldots+\alpha_d)+\delta (\beta_1+\ldots+\beta_d)},\quad \alpha,\beta \in\N_0^{d}\big\}
$$
where $m\in\R$, $0<\rho\leq 1$, and $0\leq \delta <1$. Clearly, if $\supp \mathcal F^s \sigma \subseteq M$ and $\sigma\in S^m_{\rho,\delta}$, then $L_\sigma\in OPW^{\infty,\infty}_{1\otimes w_{s}}(M)$ if $s\leq -m$ and $L_\sigma\in OPW^{\infty,q}_{1\otimes w_{s}}(M)$ if $(m+s)q< -1$.
\end{remark}

\begin{theorem} \label{thm:OPWextend} Let $1\leq p,q\leq \infty$ and $w$ moderate on $\R^2$. For $M$  compact exists $C>0$ with
\begin{eqnarray}
  \|Hf\|_{M_w^{p,q}}\leq C\, \|\sigma_H\|_{L_w^{p,q}}\, \|f\|_{M^{\infty,\infty}},
  \quad H\in OPW_w^{p,q}(M),\ f\in M^{\infty,\infty}(\R^d).\notag %  \label{eqn:OPWboundedOnMod}.
\end{eqnarray}
 Consequently, any $\displaystyle H\in OPW_w^{p,q}(M)$ extends to a bounded operator mapping  $M^{\infty,\infty}(\R^d)$ to $M_w^{p,q}(\R^d)$.
\end{theorem}
\begin{proof} Set $\omega(x,\xi,\nu,t)=w(x,\xi+\nu)$ and choose $\varphi\in\mathcal S(\R^d)$ with $\supp \widehat \varphi \subseteq [-1,1]^d$. Then we use Lemma~\ref{lemma:compactbounded} and $\supp V_{\varphi{\otimes}\varphi} \sigma_H\subseteq \R^{2d}\times M+[-1,1]^{2d}$, hence, $\omega \asymp w{\otimes}1$ on $\supp V_{\varphi{\otimes}\varphi} \sigma_H$,  to obtain
 \begin{eqnarray}
   \| \sigma_H\|_{L_w^{p,q}}&\asymp &\| \sigma_H\|_{M_{w\otimes 1}^{(p,q),(1,1)}}\asymp \| w{\otimes}1\, V_{\varphi{\otimes}\varphi} \sigma_H\|_{L^{(p,q),(1,1)}}
\asymp   \| \omega \, V_{\varphi{\otimes}\varphi} \sigma_H\|_{L^{(p,q),(1,1)}} \asymp  \| \sigma_H\|_{M_{\omega}^{(p,q),(1,1)}}\notag\,.
 \end{eqnarray} An application of Theorem~\ref{theorem:boundedoperators} with $p_1=q_1=\infty$, that is $p_1'=q_1'=1$, $p_2=p_3=p$, $q_2=q_3=q$, and $p_4,q_4=1$  concludes the proof.
\end{proof}

In the following, we set  $Q_{T}=[0, T_1) {\times}\ldots{\times} [0, T_d)$ for $T=(T_1,\ldots, T_d)\in(\R^+)^d$ and $R_{\Omega}=[-\frac {\Omega_1} 2, \frac {\Omega_1} 2){\times}\ldots{\times} [-\frac {\Omega_d} 2, \frac {\Omega_d} 2)$ for $\Omega=(\Omega_1,\ldots, \Omega_d)\in(\R^+)^d$.

\begin{theorem}\label{thm:main-full} Let $1\leq p,q\leq \infty$ and let $w=w_1{\otimes}w_2$  be moderate on $\R^{2d}$.  Let $T,\Omega\in (\R^+)^d$ satisfy  $T_m\Omega_m< 1$, $m=1,\ldots,d$. Let $\Lambda=T_1\Z\times\ldots\times T_d\Z$ and choose  $s \in M^{1,1}(\R^d)$ with $\supp \widehat s \subseteq R_{1/T}$ and $\widehat s \equiv 1$ on $R_\Omega$. Then
  \begin{eqnarray}
    \|H\sum_{\lambda \in\Lambda}\delta_{\lambda}\|_{M_w^{p,q}}\asymp \|H\|_{OPW_w^{p,q}},\quad H\in OPW_w^{p,q}( Q_T{\times}R_\Omega),\label{eqn:operatorstability-full}
  \end{eqnarray} and  any $H\in OPW_w^{p,q}( Q_T{\times}R_\Omega)$ can be reconstructed by means of
  \begin{eqnarray}
    \kappa_H(x+t,x)= \chi_{Q_{T}}(t) \sum_{\lambda \in\Lambda} \big(H\sum_{\lambda'\in\Lambda}\delta_{\lambda'}\big)(t+\lambda) \, s(x-\lambda).  \label{eqn:operatorreconstruction-full}
  \end{eqnarray}
  with convergence in $OPW_w^{p,q}(\R^{2d})$ for $1\leq p,q<\infty$ and weak-$\ast$ convergence else.
\end{theorem}

\begin{proof}  We shall show $\eqref{eqn:operatorreconstruction-full}$. The norm equivalence \eqref{eqn:operatorstability-full} can be shown by adopting the steps of the proof of  Theorem~\ref{thm:main-full2}.

For $\Lambda=T_1\Z\times\ldots\times T_d\Z$, we consider the Zak transform given by
\begin{eqnarray}
   Z_\Lambda f(t,\nu)= \sum_{\lambda \in\Lambda} f(t-\lambda)\, e^{2\pi  i \lambda
        \nu}, \quad (t,\nu)\in Q_T{\times} R_{\frac 1 T}. \notag
\end{eqnarray}

Note
$\displaystyle
    \big(H\sum_{\lambda'\in\Lambda}\delta_{\lambda'}\big) (x)
        =\langle \kappa_H(x,\cdot),\sum_{\lambda' \in\Lambda}\delta_{\lambda'}\rangle
        =\sum_{\lambda'\in\Lambda}\kappa_H(x,\lambda')=\sum_{\lambda'\in\Lambda}h_H(x,x-\lambda').
$
We consider first $h_H\in M^{1,1}(\R^{2d})$ and use the  Tonelli--Fubini Theorem and the Poisson Summation Formula \cite{Gro01},
page 250, to obtain for $(t,\nu)\in Q_{T,\frac 1 T}$
\begin{eqnarray*}
    Z_\Lambda {\circ} H \sum_{\lambda' \in\Lambda}\delta_{\lambda'} (t,\nu)
        &=&  \sum_{\lambda \in \Lambda}\sum_{\lambda'\in\Lambda} h_H(t-\lambda,t-\lambda-\lambda')  e^{2\pi i \lambda \nu}
            \\
        &=&  \sum_{\lambda \in \Lambda}\sum_{\lambda'\in\Lambda}
            \int \eta_H(t-\lambda-\lambda', \nu') e^{2\pi i ( t-\lambda)\nu'}\, d\nu'\, e^{2\pi i \lambda \nu}
            \\
       % &\stackrel{\begin{smallmatrix}
%                \xi=\nu+\weight \\
%                n=k+l \\
%              \end{smallmatrix} }{=}&
          &=&  \sum_{\lambda \in \Lambda}\sum_{\lambda'\in\Lambda}
            \int \eta_H(t-\lambda-\lambda', \nu''+\nu) e^{2\pi i (( t-\lambda)(\nu+\nu'')+ \lambda \nu)} \,d\nu''
            \\
            &=& e^{2\pi i t \nu}  \sum_{\lambda''\in\Lambda}\sum_{\lambda \in \Lambda}
            \int \eta_H(t-\lambda'', \nu+\nu'') e^{2\pi i t \nu''}\, e^{-2\pi i \lambda \nu''} \,d\nu''
                 \\
            &=& e^{2\pi i t \nu} \sum_{\lambda''\in\Lambda} \sum_{\lambda \in \Lambda^\perp}
            \eta_H(t-\lambda'', \nu+\lambda) e^{2\pi i t \lambda }
            \\
           &=&  \sum_{\lambda''\in\Lambda} \sum_{\lambda \in \Lambda^\perp}
            \eta_H(t-\lambda'', \nu-\lambda) e^{2\pi i t(\nu-\lambda)},
\end{eqnarray*}
where $\Lambda^\perp=\{\lambda\in \R^d: e^{2\pi i \lambda \lambda'}=1 \text{ for all } \lambda'\in\Lambda\}= \frac 1 {T_1}\Z \times \ldots \times \frac 1 {T_d} \Z$ is the dual lattice of $\Lambda$.

This leads directly to \eqref{eqn:operatorreconstruction-full} since
\begin{eqnarray*}
   \int &\chi_{Q_{T}}(t) & \sum_{\lambda\in\Lambda} \big(H\sum_{\lambda'\in\Lambda}\delta_{\lambda'}\big)(t+\lambda)\, s(x-\lambda)\ e^{-2\pi i
    \nu x}\ dx\\
   &=&\chi_{Q_{T}}(t) \sum_{\lambda\in\Lambda}\,\big(H\sum_{\lambda'\in\Lambda}\delta_{\lambda'}\big)(t+\lambda)\,\int   s(x-\lambda)\ e^{-2\pi i
   \nu x }\,dx\\
   &=&\chi_{Q_{T}}(t) \sum_{\lambda\in\Lambda}\,\big(H\sum_{\lambda'\in\Lambda}\delta_{\lambda'}\big)(t+\lambda)\   e^{-2\pi i
   \lambda\nu}\, \widehat{s}(\nu)\\
   &=&\chi_{Q_{T}}(t) \widehat{s}(\nu)\ \left( Z_\Lambda
   \big(H\sum_{\lambda'\in\Lambda}\delta_{\lambda'}\big)\right)(t,\nu)\\&=&\eta_H(t,\nu)e^{2\pi i \nu t}=\int h_H(x,t)\ e^{-2\pi i
    \nu(x-t)}\,dx=\int h_H(x+t,t)\ e^{-2\pi i
    \nu x}\,dx.
\end{eqnarray*}

We can apply Lemma~\ref{lemma:compactbounded} to show that $\|H\|_{OPW_w^{p,q}}\asymp\|h_H\|_{ M_{\widetilde w}^{(p,1),(1,q)}}$, $\widetilde w (x,t,\nu,\xi)=w(x,\xi)$, and validity of \eqref{eqn:operatorreconstruction-full} for $h_H\in M_{\widetilde w}^{(p,1),(1,q)}(\R^{2d})$ follows then from the density of
$M_{\widetilde w}^{1,1}(\R^{2d})$ in $M_{\widetilde w}^{(p,1),(1,q)}(\R^{2d})$. In case of $p=\infty$ or $q=\infty$ it
follows from weak-$\ast$ density.
\end{proof}
\begin{remark}\rm
  If  $p,q=2$, then we can alternatively choose $s=\chi_{R_\Omega} \in M^{2,2}(\R^d)=L^2(\R^d)$. This allows us to replace the inequality $T_m\Omega_m< 1$ by  $T_m\Omega_m\leq 1$ in the hypothesis of Theorem~\ref{thm:main-full}.
\end{remark}
Note that Theorem~\ref{thm:main-full} and its proof  generalize trivially to the following setting.
\begin{theorem}\label{thm:main-fullb}
Let $1\leq p,q\leq \infty$ and $w=w_1{\otimes}w_2$ be moderate on $\R^{2d}$.  Let $A,B\subseteq \R^d$ be bounded, and let $\Lambda$ be a lattice such that $A$ is contained in a fundamental domain of $\Lambda$ and for some $\epsilon>0$, $B+[-\epsilon,\epsilon)^d$ is contained in a bounded fundamental domain of $\Lambda^\perp=\{\lambda\in \R^d: e^{2\pi i \lambda \lambda'}=1 \text{ for all } \lambda'\in\Lambda\}$.
Choose  $s \in M^{1,1}(\R^d)$ with $\supp \widehat s \subseteq B+[-\epsilon,\epsilon)^d$ and $\widehat s \equiv 1$ on $B$. Then
  \begin{eqnarray}
    \|H\sum_{\lambda \in\Lambda}\delta_{\lambda}\|_{M_w^{p,q}}\asymp \|H\|_{OPW_w^{p,q}},\quad H\in OPW_w^{p,q}(A{\times}B), \notag %\label{eqn:operatorstability-full}
  \end{eqnarray} and  any $H\in OPW_w^{p,q}(A{\times}B)$ can be reconstructed by means of
  \begin{eqnarray}
    \kappa_H(x+t,t)= \chi_{A}(t) \sum_{\lambda \in\Lambda} \big(H\sum_{\lambda'\in\Lambda}\delta_{\lambda'}\big)(t+\lambda)\, s(x-\lambda).  \notag %\label{eqn:operatorreconstruction-full}
  \end{eqnarray}
  with convergence in $OPW_w^{p,q}(A{\times}B)$ for $1\leq p,q<\infty$ and weak-$\ast$ convergence else.
\end{theorem}

Considering $OPW^{p,\infty}([0,T){\otimes} [-\tfrac \Omega 2, \tfrac \Omega 2))$, we obtain the classical sampling theorem as corollary to Theorem~\ref{thm:main-full}.
\begin{corollary}\label{thmimproved-sampling}
For $m\in L^p(\R)$, $1\leq p\leq \infty$, with $\supp \widehat m \subseteq [-\tfrac \Omega 2, \tfrac \Omega 2)$ and  $T$ with $T\Omega< 1$ choose $s\in M^{1,1}(\R)$ with $\supp \widehat s \subseteq  [-\tfrac \Omega 2, \tfrac \Omega 2)$ and $\widehat s \equiv 1$ on  $[-\tfrac 1 {2T}, \tfrac 1 {2T})$. Then
\begin{eqnarray}
  \|m\|_{L^p}\asymp\| \left\{m(kT)\right\}
\|_{l^p}\label{equation:improved-samplingnorms}
\end{eqnarray}
and
  \begin{eqnarray}
        m(x)= \sum_{k\in\Z} m(kT)\,
        s(x-kT).\notag %\label{equation:improved-sampling}.
  \end{eqnarray}
\end{corollary}

\begin{proof}
For $m\in L^p(\R)$ with $\supp \widehat m \subseteq [-\tfrac \Omega 2, \tfrac \Omega 2)$, we define the multiplication operator $M$ formally by $M:f\longrightarrow m\cdot f$, $f\in \mathcal S(\R)$. We have
\begin{eqnarray*}
     M f(x)     &=& m(x)f(x)=\int m(x)\delta_0(t)f(x-t)dt
                = \iint \delta_0(t)\fhat{m}(\nu)e^{2\pi i x \nu}f(x-t)\,dt \,d\nu\,.
\end{eqnarray*}
Hence,  $\delta_0\otimes \fhat{m}=\eta_M=\mathcal F^s\sigma_M$, and, picking any $T$ with $T\Omega<1$, we conclude $M \in OPW^{p,\infty}\big([0,T){\times}[-\tfrac \Omega 2, \tfrac \Omega 2)\big)$.

Theorem~\ref{thm:main-full} implies that $H$ and therefore $m$ is fully recoverable from $M \sum_{k\in\Z}\delta_{kT}=
\sum_{k\in\Z} m(kT)\delta_{kT}$, in fact, the reconstruction formula \eqref{eqn:operatorreconstruction-full} reduces then to the classical reconstruction formula for functions:
\begin{eqnarray*}
 m(x)\delta_0(t)&=&m(x+t)\delta_0(t) =\kappa_M(x+t,x)
   \stackrel{{\rm Thm~\ref{thm:main-full}}}{=} \chi_{[0,T)}(t)\sum_{n\in\Z}\, \big(M\sum_{k\in\Z}\delta_{kT}\big)(t+nT)\, s(x-nT) \\
        &=&\chi_{[0,T)}(t)\sum_{n\in\Z}\sum_{k\in\Z}m(kT)\delta_{kT}(t+nT)s(x-nT)\\
        &=&\chi_{[0,T)}(t)\sum_{n\in\Z}\sum_{k\in\Z}m(kT)\delta_{0}(t-(kT-nT))s(x-nT)\\
        &=&\left\{%
\begin{array}{ll}
    0, & \hbox{if $t\notin [0,T)$ or $t \notin \Z T$;} \\[.3cm]
    \sum_{n\in\Z}\sum_{k\in\Z}m(kT)\delta_{0}((n-k)T)s(x-nT), & \hbox{\ } \\[.1cm]
    \quad =\ \sum_{k\in\Z}m(kT)s(x-kT), & \hbox{if $t=0$.} \\
\end{array}%
\right.\\
&=&\delta_0(t)\sum_{k\in\Z}m(kT)s(x-kT)\,.
\end{eqnarray*}

The norm equivalence in \eqref{equation:improved-samplingnorms} is obtained by verifying that
\begin{eqnarray*}
 \|m\|_{L^p}&\asymp & \|M \|_{OPW^{p,\infty}}
\asymp  \| M \sum_{n\in\Z}\delta_{nT}\|_{M^{p,\infty}}
    = \|\sum_{n\in\Z} m(nT)\delta_{nT}\|_{M^{p,\infty}}
    \asymp  \|\{m(nT)\}_n\|_p, \\  && \hspace{6cm} m\in L^p(\R), \  \supp \widehat{m}\subseteq [-\frac \Omega 2,\frac \Omega 2]\,.
\end{eqnarray*}
%
%Further, we have
%\begin{eqnarray*}
%\|h\|_{M^{1p\infty p}} &=&\|\delta_0\otimes m\|_{M^{1p\infty
%p}}=\|\delta_0\|_{M^{1q}}\|m\|_{M^{pp}}
%\end{eqnarray*}
%so we conclude that in general
%$$\|m\|_{M^p}\asymp \|h\|_{M^{1p\infty p}}\asymp
%\|H\dtrain_T\|_{M^{p\infty}}\asymp \|\{m(nT)\}_n\|_p.$$ In the case
%$p=2$, we have
%$$
%    \|m\|_{L^2}\asymp\|m\|_{M^2}\asymp \|\{m(nT)\}_n\|_2,
%$$
%which completes the proof of
%Theorem~\ref{thmimproved-sampling}
\end{proof}

In addition to the application of Theorem~\ref{thm:main-full} to multiplication operators, we consider now \\ $OPW^{\infty,p}([0,T){\otimes} [-\tfrac \Omega 2, \tfrac \Omega 2))$ for convolution operators.

\begin{example}\label{exa:convolution}\rm
Time invariant operators are convolution operators,
that is,
$$\displaystyle Hf(x)=h\ast f(x)=\int h(x-s) f(s) \, ds
%, \quad \supp h\subset [0, 1]
.
$$ Such operators represent the classical example of operator identification/sampling namely, as $H\delta_0
(x) = h(x)$, $H\delta_0$ determines $h$ and therefore $H$ completely.   In the framework of operator sampling, we consider $h\in L^p(\R)$ with $\supp h\subseteq [0,T]$. We have $\eta_H(t,\nu)=\mathcal F^s \sigma_H(t,\nu)=h(t)\delta_0(\nu)$ and $H\in OPW^{\infty,p}\big([0,T){\times}[-\frac 1 {4T},\frac 1 {4T})\big)$. Moreover, with appropriate $s$ we may obtain $H\delta_0=h$ from \eqref{eqn:operatorreconstruction-full}, as
\begin{eqnarray*}
 h(t)&=&h_H(x,t)=h_H(x+t,t)
    \stackrel{{\rm Thm~\ref{thm:main-full}}}{=}\chi_{[0,T)}(t) \sum_{n\in\Z}\, \big(H\sum_{k\in\Z}\delta_{kT}\big)(t+nT)\, s(x-nT) \\
    &=& \chi_{[0,T)}(t) \sum_{n\in\Z} \sum_{k\in\Z}h(t-(kT-nT))\, s(x-nT)
    = \sum_{n\in\Z} \sum_{k\in\Z}\chi_{[0,T)}(t) h(t-(kT-nT))\, s(x-nT) \\
    &=& \sum_{n\in\Z} H\delta_0(t)s(x-nT) =  H\delta_0(t)\sum_{n\in\Z}s(x-nT)
      =  H\delta_0(t)\sum_{\ell \in\Z}\widehat{s}(\tfrac \ell T)e^{2\pi i x \ell}
           =  H\delta_0(t)\widehat{s}(0)=H\delta_0(t)\,.
\end{eqnarray*}
The
distributional spreading support of a time invariant operator is also
indicated in Figure~\ref{fig:niklasPic}.
\end{example}

\section{Necessary and sufficient conditions for operator sampling and identification}\label{section:MainResults2}
The aim of this section is to show that the applicability of sampling methods for operators depends solemnly on the size of the spreading support set $M$, that is, on the Jordan content of $M$ (see Definition~\ref{def:JordanContent} below). Our main result in this section, namely Theorem~\ref{thm:main-full2}, though, only covers the case $d=1$, that is, operators  $H:\mathcal S(\R)\longrightarrow \mathcal S'(\R)$. Possible means for generalizing Theorem~\ref{thm:main-full2} to operators $H:\mathcal S(\R^d)\longrightarrow \mathcal S'(\R^d)$ are briefly discussed in Section~\ref{section:outlook}.

Before recalling the definition of Jordan domains and some of their properties, and before stating and proving Theorems~\ref{thm:main-full2} and \ref{thm:main-full3}, we will use a geometric approach to obtain a sufficient condition for the identifiability of $OPW_w^{p,p}(M)$ if $M= A ( Q_T{\times}R_\Omega)+(t_0,\nu_0) \subseteq {\R^{2d}}$, $T,\Omega\in (\R^+)^d$,  $T_m\Omega_m< 1$, $m=1,\ldots,d$, and $A$ is a so-called symplectic matrix.  Theorem~\ref{thm:symplectic} below generalizes Theorem 5.4 in \cite{KP06}.

\begin{definition}
  The symplectic group $Sp(d,\R)$ consists of those matrices  $A\in SL(2d,\R)=\{A\in R^{2d\times 2d}: \ \det A=1\}$ with $ A^\star \left(
\begin{smallmatrix}
0&-I_d\\
I_d&0
\end{smallmatrix}
\right)
A
=
\left(
\begin{smallmatrix}
0&-I_d\\
I_d&0
\end{smallmatrix}
\right), $ where $I_d$ is the $d{\times}d$ identity matrix.
\end{definition}
Note that $A\in Sp(d,\R)$ if and only if $[A(x,\xi)^T,A(x',\xi')T]=[(x,\xi),(x',\xi')]$ where $[\cdot,\cdot]$ is the symplectic form defined in Section~\ref{section:notation}.

\begin{theorem} \label{thm:symplectic} Let  $A\in Sp(d,\R)$, $(t_0,\nu_0)\in \R^{2d}$, $1\leq p \leq \infty$, and let $w$ be a moderate weight  on $\R^{2d}$ with $w(A(x,\xi)^T)\leq w(x,\xi)$. Then

\begin{enumerate}
  \item $OPW_w^{p,p}(M)$ mapping $M^{\infty,\infty}(\R^d)$ to $M_w^{p,p}(\R^d)$  is identifiable if and only if $OPW_w^{p,p}\big(AM+(t_0,\nu_0)\big)$ mapping $M^{\infty,\infty}(\R^d)$ to $M_w^{p,p}(\R^d)$ is identifiable, and, consequently,
      \item for $T,\Omega\in (\R^+)^d$ with  $T_m\Omega_m< 1$, $m=1,\ldots,d$, we have  $ OPW_w^{p,p}\big( A ( Q_T{\times}R_\Omega)+(t_0,\nu_0)\big)$  mapping $M^{\infty,\infty}(\R^d)$ to $M_w^{p,p}(\R^d)$ is identifiable.
\end{enumerate}

\end{theorem}
The proof of Theorem~\ref{thm:symplectic} is based on the representation theory of the Weyl-Heisenberg group. Here, we only outline the proof, the interested reader can import details from Section 5 in \cite{KP06} or \cite{Fol89,Gro01}.
\begin{proof}
We will obtain the identifiability
$OPW_w^{p,p}(AM)$ with $A\in Sp(2d,\R)$ from the identifiability of $OPW_w^{p,p}(M)$ by
using the canonical correspondence between elements in $OPW_w^{p,p}(AM)$
and elements in $OPW_w^{p,p}(M)$ which is given by a coordinate transformation in
the spreading domain $\R^{2d}\supseteq M, AM$. In fact, Theorem 5.3 in \cite{KP06}
recalls that for  $A\in Sp(d,\R)$, there exists a unitary operators
$O_A$ on $L^2(\R^d)$  with $\pi(A (t,\nu))=O_A \pi(t,\nu){O_A}^{\ast}$, $t,\nu\in\R^d$. Such operators $O_A$, $A\in Sp(d,\R)$ are called metaplectic operators, and they are  intertwining operators for representations of the reduced Weyl--Heisenberg group that are unitarily equivalent to the Schr\"odinger representation \cite{Fol89,Gro01}. Metaplectic operators are finite compositions of the Fourier transform, multiplication operators with multiplier $e^{-\pi i x^T C x}$ with $C$ selfadjoint, and normalized dilations $f\mapsto |\det D|^{\frac 1 2}\, f(Dx)$, $D$ invertible. They extend, respectively restrict, to isomorphisms on $M_w^{p,p}(\R^d)$, $1\leq p\leq \infty$, if $w(A(x,\xi)^T)\leq w(x,\xi)$ (see Theorem 7.4 in \cite{FG92a}).

The following formal
calculations of operator valued  integrals can be justified weakly for all
$H\in OPW_w^{p,p}(AM)$. A similar computation can be made for $H\in OPW_w^{p,p}\big(M+(t_0,\nu_0)\big)$, combining both leads Theorem~\ref{thm:symplectic}. We compute
\begin{eqnarray}
\qquad  H &=&  \iint \eta_H(t,\nu)\, M_\nu T_t \ dt\,d\nu
    =  \iint  \eta_H(t,\nu)\  \pi(t,\nu) \,dt\,d\nu \nonumber\\
    &=&  \iint \eta_H(A (t,\nu)^T)\  \pi(A (t,\nu)^T) \ dt\,d\nu=  \iint \eta_H(A (t,\nu)^T)\  O_A\pi(t,\nu){O_A}^{\ast} \ dt\,d\nu \nonumber\\
    &=&  O_A\iint \eta_{H_A}(t,\nu)\  \pi(t,\nu) \ dt\,d\nu\ {O_A}^{\ast}
    =  O_A\ H_A \ {O_A}^{\ast},
   \notag % \label{equ:coordinatetransformation}
   % \label{equation:formal}
\end{eqnarray}
where $\eta_{H_{A}}=\eta_H{\circ}A$ and $H_{A}\in OPW_w^{p,p}(M)$. The
identifiability of $OPW_w^{p,p}(M)$ with identifier $f_M\in M^{\infty,\infty}(\R^d)$ leads therefore to the identifiability
of $OPW_w^{p,p}(AM)$ with identifier
$f_{AM}=O_A f_M\in M^{\infty,\infty}(\R^d)$.  In fact, we have
\begin{eqnarray}
   \|H  f_{AM}\|_{M_w^{p,p}}
        &=& \|H \, O_A f_M\|_{M_w^{p,p}}\asymp\|{O_A}^{\ast} H O_A f_M \|_{M_w^{p,p}} \notag \\
        &=&\|H_{A} f_M\|_{M_w^{p,p}}
        \asymp \|\sigma_{H_{A}} \|_{L_w^{p,p}}
        \asymp \|\sigma_{H} \|_{L_w^{p,p}}
        =\|H\|_{OPW_w^{p,p}},\quad H\in OPW_w^{p,p}(AM)\notag.
\end{eqnarray}
\end{proof}
\begin{remark}
\rm Theorem~\ref{thm:symplectic} is not an operator sampling result per se as not necessarily all $O_A$ map discretely supported distributions to discretely supported distributions. But  Theorem~\ref{thm:symplectic} can be used to show that  $OPW^{p,p}_w(M)$ permits operator sampling by showing that
\begin{enumerate}
\item $M\subseteq A\widetilde M + (t_0,\nu_0)$,

   \item $A\in Sp(d,\R)$,
   \item $w(A(x,\xi))\leq w(x,\xi)$, %$x,\xi\in\R^d$,
   \item $OPW^{p,p}_w(\widetilde M)$ permits operator sampling with sampling set $\{x_j\}$ and weights $\{c_j\}$, and
   \item $O_A^\ast \sum c_j\delta_{x_j}$ is discretely supported.
 \end{enumerate}
 Note also that the restriction to $p=q$ in Theorem~\ref{thm:symplectic} is necessary, as, for example, the Fourier transform is not an isomorphism on $M^{p,q}$ whenever $p\neq q$.
\end{remark}

Theorem~\ref{thm:symplectic} relies on arguments based on  symplectic geometry on phase space. As discussed above, Theorems~\ref{thm:main-full2} and \ref{thm:main-full3} give a characterization for the identifiability of operators $H:\mathcal S(\R)\longrightarrow \mathcal S'(\R)$ which does not rely on any geometric properties.

\begin{definition}\label{def:JordanContent}
For $K,L\in\N$ set $R_{K,L}=[0,\tfrac 1 K )\times[0,\tfrac{K}{L})$
and
\begin{eqnarray*}
  \mathcal{U}_{K,L}
        =   \Big\{
                \bigcup_{j=1}^J \Big(R_{K,L}+(\tfrac {k_j} K  , \tfrac {p_j K} {L}   )
                \Big):
                \ k_j,p_j\in\Z, J\in\N
            \Big\}. \label{equation:U-KL}
\end{eqnarray*}
The inner content, respectively outer content, of a bounded set
$M\subseteq \R^2$  is
\begin{eqnarray}
 \vol^-(M)= \sup\{ \mu (U):&&\hspace{-.6cm} U{\subseteq} M\text{ and } U\in  \mathcal
U_{K,L} \text{ for some } K{,}L\in \N\}\, ,
\label{equation:innercontent}
\end{eqnarray} respectively
\begin{eqnarray}
\vol^+(M)=\inf\{ \mu(U):&&\hspace{-.6cm} U{\supseteq}
M\text{ and } U\in  \mathcal U_{K,L} \text{ for some }
K{,}L\in \N\}. \label{equation:outercontent}
\end{eqnarray}

Clearly,  we have $\vol^-(M)\leq \vol^+(M)$. If $\vol^-(M)=
 \vol^+(M)$, then we say that $M$ is a  Jordan domain
with Jordan content $\vol(M)=\vol^-(M)=
 \vol^+(M)$.
\end{definition}

We collect some well known and useful facts on Jordan domains to illustrate their generality \cite{Fol99}.
\begin{proposition}\label{prop:jordancontent} Let $M\subseteq \R^2$.
\begin{enumerate}
  \item If $M$ is a Jordan domain, then $M$ is Lebesgue
  measurable with $\mu(M)=\vol(M)$.
  \item If $M$ is Lebesgue measurable and bounded and its boundary $\partial
  M$ is a Lebesgue zero set, that is, $\mu(\partial M)=0$, then $M$ is a  Jordan domain.
  \item If $M$ is open, then $\vol^-(M)=\mu(M)$
  and if $M$ is compact, then $\vol^+(M)=\mu(M)$.
  \item If $\mathcal P \subseteq\N$ is unbounded, then replacing the quantifier
  ``\,{\em for some} $L\in\N$"
  with ``{\em for some} $L\in \mathcal P$" in
  (\ref{equation:innercontent}) and in
  (\ref{equation:outercontent}) leads to  equivalent definitions of
  inner and outer Jordan content.
\end{enumerate}
\end{proposition}

The second main result of this paper has been stated in simple form as Theorem~\ref{thm:main-simple2}, part 1, in Section~\ref{section:introduction}. It also generalizes Theorem 1.1 in \cite{PW06b}.
\begin{theorem}\label{thm:main-full2}
For $1\leq p,q\leq \infty$ and $w=w_1\otimes w_2$ moderate, the class $OPW_w^{p,q}(M)$ mapping $M^{\infty,\infty}(\R)$ to $M_w^{p,q}(\R)$ permits operator sampling if $\vol^+(M)<1$. In fact, if $\vol^+(M)<1$, then there exists $L>0$ and a periodic sequence $\{c_n\}$ such that
\begin{eqnarray}
\|H\sum_{n\in\Z} c_n \delta_{\frac n L}\|_{M_w^{p,q}} \asymp \|H\|_{OPW_w^{p,q}},  \quad H\in OPW_w^{p,q}(M)\,.
 \label{eqn:mainresult2item1}
\end{eqnarray}
\end{theorem}
Theorem~\ref{thm:main-full2} is complemented by Theorem~\ref{thm:main-full3} which generalizes Theorem 1.1 in \cite{PW06b} and Theorem 5.2, part 2, in \cite{PW06}.
\begin{theorem}\label{thm:main-full3}
Let $1\leq p,q\leq \infty$ and $w$ subexponential. The class $OPW_w^{p,q}(M)$ mapping $M^{\infty,\infty}(\R)$ to $M_w^{p,q}(\R)$ is not identifiable if
$\vol^-(M)>1$, that is, for all $f\in M^{\infty,\infty}$ we have
\begin{eqnarray}
  \|Hf\|_{M_w^{p,q}}\not \asymp  \|H\|_{OPW_w^{p,q}} \quad H\in OPW_w^{p,q}(M)\,.
  \notag
  %\label{eqn:mainresult2item2}
\end{eqnarray}
\end{theorem}

Theorem~\ref{thm:main-full2} is proven below. Subsequently, we outline the proof of Theorem~\ref{thm:main-full3} which employs elements of the proof of Theorem 1.1 in \cite{PW06b} and Theorem 3.13 in \cite{Pfa08b}.

\subsection{Proof of Theorem~\ref{thm:main-full2}}

The following observations are special cases of Theorem~\ref{thm:symplectic}. They will be used in the following to  reduce  notational complexity.

\begin{proposition}\label{prop:equivalentIdentifiability} Let $1\leq p,q \leq \infty$ and let $w$ be a moderate weight on $\R^2$.
  \begin{enumerate}
    \item $OPW^{p,q}_w(M)$ is identifiable by $f$ if and only if $OPW^{p,q}_w\big( M{\cdot}\left(\begin{smallmatrix}
  1/a  &    0           \\
  0  & a
  \end{smallmatrix}\right)\big)$ is identifiable by $D_af:x\mapsto f(a x)$.%, $x\in\R$.
  \item $OPW^{p,q}_w(M)$ is identifiable by $f$ if and only if $OPW^{p,q}_w\big(M+\lambda\big)$ is identifiable by $\pi(\lambda)f$.
  \end{enumerate}
\end{proposition}

\begin{proof}
We shall proof {\it 1.}, the proof of {\it 2.} follows similarly. For $H\in OPW^{p,q}_w\big( M{\cdot}\left(\begin{smallmatrix}
  1/a  &    0           \\
  0  & a
  \end{smallmatrix}\right)\big)$, define $H_a\in OPW^{p,q}_w ( M)$ by $\eta_{H_a}=\eta_H{\circ} \left(\begin{smallmatrix}
  1/a  &    0           \\
  0  & a
  \end{smallmatrix}\right)$. Then $\sigma_{H_a}=\sigma_H{\circ} \left(\begin{smallmatrix}
  1/a  &    0           \\
  0  & a
  \end{smallmatrix}\right)$ as well. We compute formally
  \begin{eqnarray}
    \big(HD_a f\big)(\tfrac x a)&=& \int \sigma_H (\tfrac x a, \xi)\, e^{2\pi i \frac x a \xi} \, \widehat{D_a f}(\xi)\, d\xi = \tfrac 1 a \int \sigma_H (\tfrac x a, \xi) \,e^{2\pi i  x  \frac \xi a} \,\widehat{f}(\tfrac \xi a)\, d\xi \notag \\
&=&  \int \sigma_H (\tfrac x a, a \xi)\, e^{2\pi i x \xi} \,\widehat{f}(\xi)\, d\xi \notag=H_af(x).
  \end{eqnarray}
  Using standard density arguments, we conclude that
  \begin{eqnarray}
    \|H D_a f\|_{M_w^{p,q}} {\asymp} \|D_{\frac 1 a} H D_a f\|_{M_w^{p,q}}{=}\|H_a f\|_{M_w^{p,q}}{\asymp}\|H_a \|_{OPW_w^{p,q}}{\asymp}\|H \|_{OPW_w^{p,q}},\ \ H\in OPW^{p,q}_w\big( M{\cdot}\left(\begin{smallmatrix}
  1/a  &    0           \\
  0  & a
  \end{smallmatrix}\right)\big) \,. \notag
  \end{eqnarray}

\end{proof}

Assume now that $\vol^+(M)<1$.  Applying Proposition~\ref{prop:equivalentIdentifiability}, we assume, without loss of generality, that for $\delta>0$ sufficiently small, $M+[-\tfrac \delta 2, \tfrac
\delta 2)^2\subseteq [0,1)\times \R^+$. We choose $K,\,L\in\N$ with $L$ prime so that the following conditions are satisfied for some
$0<\epsilon <\delta$ and $M_\epsilon=M+[-\tfrac \epsilon 2, \tfrac
\epsilon 2)^2$
\begin{enumerate}
\item $\vol^+(M_\epsilon)<1$,

\item $M_\epsilon \subseteq[0,1)\times[0,K)$,

\item  $L\ge K$,

\item $\displaystyle M_\epsilon \subseteq U_M=
                \bigcup_{j=0}^{L-1} \left(R_{K,L}+(\tfrac {m_j} K  , \tfrac {n_j K} {L}   )
                \right),\
                \ m_j,n_j\in\Z,$ where $R_{K,L}=[0,\frac 1 K)\times[0,\frac K L)$ and
$(m_j,n_j)\ne(m_{j'},n_{j'})$ if $j\ne j'$.
\end{enumerate}
Note that
 $1= \vol(U_M)$.

The following result from \cite{LPW05} is a key component of our proof of Theorem~\ref{thm:main-full2}. In fact, if the restriction to $L$ prime below could be weakened, then we would obtain a generalization of Theorem~\ref{thm:main-full2} to higher dimensions (see Section~\ref{section:outlook}).
\begin{theorem}\label{prop:generallinear}
  For $c\in\C^L$ define $\pi(k,\ell)c$ by $(\pi(k,\ell)c)_j=c_{j-k}\,e^{2\pi i \frac {j\ell}L}$, $k,\ell=0,\ldots, L-1$, where $j-k$ is understood modulus $L$.  If $L$ is prime, then for almost every $c\in \C^L$, the vectors in  $\mathcal G_c=\{\pi(k,\ell)c \}_{k,\ell=0,\ldots, L-1}$ are in general linear position, that is, any matrix composed of  $L$ vectors of $\mathcal G_c$ is invertible.
\end{theorem}

\begin{remark} \rm
  Theorem~\ref{prop:generallinear} can be reformulated as a matrix identification result with identifier $c$ \cite{KPR08}. The use of algorithms based on  basis pursuit  to determine a matrix $M$ from $Mc$ efficiently is discussed in \cite{PRT08,PR09}.
\end{remark}

We now choose as $c\in\l^\infty(\Z)$ the periodic extension of a vector $(c_0,\ldots, c_{L-1})$ which satisfies the conclusions of Theorem~\ref{prop:generallinear}.
In the following, we shall show that
$\kappa_H$ can be recovered from $Hg$ with $g=\sum_{k\in\Z}c_k\delta_{\frac k L}\in M^{\infty,\infty}(\R)$.

For simplicity, we shall assume first that $\kappa_H\in M^{1,1}(\R^2)$. This additional assumption implies that for $g\in M^{\infty,\infty}(\R)$, we have $Hg\in M^{1,1}(\R)$ \cite{PW06}. This enables us to switch the order of  integration in many of the following computations. After the necessary computations are completed, we shall extend our result using standard density arguments.

Choose nonnegative $\varphi\in C^\infty_c(\R)$ with $\int\varphi(x)\,dx=1$ and $\supp\varphi\subseteq [-\tfrac \epsilon 2,\tfrac \epsilon 2)$.
We shall consider the case $1\leq p<\infty$ only, the case $p=\infty$ requires only  the usual adjustments.

Note that
\begin{eqnarray*}
 \big| V_\varphi Hg (\rtwo+\tfrac n K,\gtwo)\big|&=&\Big|\int \sum_{k\in\Z} c_k\, h(\rone,\rone+\tfrac k K)\,  e^{-2\pi i \gtwo \rone}\varphi(\rone-(\rtwo+\tfrac n K)\,d\rone\Big|\\
&=& \Big|\int \sum_{k\in\Z}c_k\,  h(\rone+\tfrac n K,\rone+\tfrac n K+\tfrac k K) \, e^{-2\pi i \gtwo \tfrac n K}\, \overline{\pi(\rtwo,\gtwo)\varphi(\rone)}\,d\rone\Big|\\
&=& \Big|\int \sum_{k\in\Z}c_k\,  h(\rone+\tfrac n K,\rone+\tfrac n K+\tfrac k K)\,  \overline{\pi(\rtwo,\gtwo)\varphi(\rone)}\,d\rone\Big|\,.
\end{eqnarray*}

Set $\weight=\sum_{m\in\Z}w_1(\frac {mL} K) \, \chi_{[\frac {mL} K,\frac {(m+1)L} K)}$ and observe that $\weight \asymp w_1$. Then
\begin{eqnarray}
  \hspace{-1cm}\|Hg\|_{M_w^{p,q}} &=& \|V_\varphi Hg\|_{L_w^{p,q}}
                  =  \Big\| \Big(\sum_{n\in\Z} \int_{\frac n K}^{\frac {n+1} K}
                     |V_\varphi Hg(\rtwo,\gtwo)w_1(\rtwo) |^p \Big)^{\frac 1 p}\, d\rtwo \Big\|_{L_{w_2}^q}\notag \\
                     &=&  \Big\| \Big(\sum_{n\in\Z} \int_0^{\frac {1} K}
                     |V_\varphi Hg(\rtwo+\tfrac n K,\gtwo)w_1(\rtwo+\tfrac n K) |^p \Big)^{\frac 1 p}\, d\rtwo \Big\|_{L_{w_2}^q}\,\notag\\
                     &=&  \Big\| \Big(\sum_{n\in\Z} \int_0^{\frac {1} K}
                     |V_\varphi Hg(\rtwo+\tfrac n K,\gtwo)\weight(\rtwo+\tfrac n K) |^p \Big)^{\frac 1 p}\, d\rtwo \Big\|_{L_{w_2}^q}\,.\notag
\end{eqnarray}

We set
$\psi=\chi_{[-\frac \epsilon 2,\frac K L +\frac \epsilon 2)}\ast\varphi$ and observe that
$\psi(\gfour)T_{\omega}
\varphi(\gfour)=T_{\omega} \varphi(\gfour)$ for $ \omega \in [0,\tfrac K L)$, a fact that will be used to drop $\psi$ in (\ref{equation:useSupportOf-g}) below. We compute for $t\in[0,\frac 1 K)$
\begin{eqnarray}
&&\hspace{-.5cm}\sum_{n\in\Z}|\weight(\rtwo+\tfrac n K) |^p
                     \Big|V_\varphi Hg(\rtwo+\tfrac{n}{K},\gtwo )\Big|^p\notag \\
&= &        \sum_{n\in\Z}|\weight(\rtwo+\tfrac n K) |^p
                     \Big|  \int \sum_{k\in\Z}c_k \, h_H(\rone+\tfrac n K,\rone+\tfrac n K+\tfrac k K) \overline{\pi(\rtwo,\gtwo)\varphi(\rone)}\,d\rone  \Big|^p\notag \\
&= & \sum_{n\in\Z}|\weight(\rtwo+\tfrac n K) |^p
                     \Big|  \int \sum_{k\in\Z}c_{n-k} \, h_H(\rone+\tfrac n K,\rone+\tfrac  k K) \overline{\pi(\rtwo,\gtwo)\varphi(\rone)}\,d\rone  \Big|^p\notag \\
&= & \sum_{j=0}^{L-1}
                    \sum_{m\in\Z} |w_1(\tfrac {m L} K) |^p\Big|  \int \sum_{k\in\Z}c_{j-k} \, h_H(\rone+\tfrac {mL+j} K,\rone+\tfrac  k K) \overline{\pi(\rtwo,\gtwo)\varphi(\rone)}\,d\rone  \Big|^p\notag \\
&= & \sum_{j=0}^{L-1}
                     \big\|   \sum_{m\in\Z} \int \sum_{k\in\Z}c_{j-k} \, h_H(\rone+\tfrac {mL+j} K,\rone+\tfrac  k K) \overline{\pi(\rtwo,\gtwo)\varphi(\rone)}\,d\rone  \notag \\
&& \hspace{5cm} V_{\varphi} (M_{-\frac {mL+j} K}\psi )(\rthree ,\gthree ) \big\|_{L^{p}_{1{\otimes}\weight}}^p\label{eqn:lpRieszSequence} \\
&\asymp & \sum_{j=0}^{L-1}
                     \big\|  \int \int \sum_{k\in\Z}c_{j-k}\Big(\sum_{m\in\Z}h_H(\rone+\tfrac {mL+j} K,\rone+\tfrac  k K) e^{-2\pi i \gfour \tfrac {mL+j}  K} \Big)\notag \\
&& \hspace{5cm}\overline{\pi(\rtwo,\gtwo)\varphi(\rone)}  \,   \, \psi(\gfour )\, \overline{\widehat{\pi}(\rthree,\gthree)\varphi(\gfour )}\,d\gfour \,d\rone \big\|_{L^{p}_{1{\otimes}w_1}}^p\label{eqn:interrupt}\,,
\end{eqnarray}
where we used that $M_{-\frac {mL+j} K}\psi$ is an $l^p_{\widetilde w}$--Riesz basis in the Banach space $M^{p,p}_{1{\otimes}w_1}(\R)$ to obtain \eqref{eqn:lpRieszSequence}, that is, we used
\begin{eqnarray*}
  \| \sum_{m\in\Z} a_m \|_{l^p_{w_1}}\asymp \| \sum_{m\in\Z} a_m M_{-\frac {mL+j} K}\psi \|_{M^{p,p}_ {w_1}}=\| \sum_{m\in\Z} a_m V_\varphi M_{-\frac {mL+j} K}\psi \|_{L^{p}_ {w_1}}\,.
\end{eqnarray*}

Note that
\begin{eqnarray}
  &&\hspace{-.5cm}\sum_{m\in\Z}h_H(\rone+\tfrac {mL+j} K,\rone+\tfrac  k K) e^{-2\pi i \gfour \frac {mL+j}  K} = \sum_{m\in\Z}\int \eta_H(\rone+\tfrac k K,\gthree) e^{2\pi i (\rone +  \frac {mL+j}  K) \gthree} \,d\gthree e^{-2\pi i\frac {mL+j}  K \gfour}  \notag \\
&=& \sum_{m\in\Z}\int \eta_H(\rone+\tfrac k K,\gthree)e^{2\pi i \rone \gthree} e^{2\pi i   \frac {mL+j}  K ( \gthree-\gfour)} \,d\gthree \notag \\
&=& \sum_{m\in\Z}\int \eta_H(\rone+\tfrac k K,\gfour+\gthree)e^{2\pi i \rone (\gfour+\gthree)}  e^{2\pi i \frac j K\gthree}  e^{2\pi i   \frac {mL}  K \gthree} \,d\gthree \notag \\
&=& \sum_{\ell\in\Z}\eta_H(\rone+\tfrac k K,\gfour+\tfrac{\ell K} L)e^{2\pi i \rone (\gfour+\frac{\ell K} L)}  e^{2\pi i   \frac {j\ell} L}   \label{eqn:usePSF} \\
&=& \sum_{\ell\in\Z}\eta_H(\rone+\tfrac k K,\gfour+\tfrac{\ell K} L)e^{2\pi i (\rone+\frac k K) (\gfour+\frac{\ell K} L)}  e^{-2\pi i(\frac k K \gfour+\frac{k\ell} L)} e^{2\pi i   \frac {j\ell} L}
 \notag \\
&=& \sum_{\ell\in\Z}\widetilde{\eta_H}(\rone+\tfrac k K,\gfour+\tfrac{\ell K} L)  e^{2\pi i   \frac {(j-k)\ell} L}e^{-2\pi i \frac k K\gfour}\,. \label{eqn:insert}
\end{eqnarray}
We have applied the Poisson Summation Formula to obtain \eqref{eqn:usePSF}. Moreover, we chose $\widetilde{\eta_H}(\rthree ,\gfour)=\eta_H(\rthree ,\gfour)e^{2\pi i \rthree\gfour}$ in \eqref{eqn:insert}.

After substituting \eqref{eqn:insert} into \eqref{eqn:interrupt}, we integrate with respect to $t$ on $[0,\tfrac 1 K)$ to obtain
\begin{eqnarray}
&&\hspace{-.5cm}\int_{0}^{\frac {1} K} \sum_{n\in\Z} \big| w_1(\rtwo+\tfrac n K,\gtwo)  \big|^p\,
                     \big|V_\varphi Hg(\rtwo+\tfrac{n}{K},\gtwo )\big|^p\,d\rtwo\notag \\
&\asymp & \sum_{j=0}^{L-1}
                     \int_{0}^{\frac {1} K} \big\|  \int \int \sum_{k\in\Z}c_{j-k}\big(\sum_{\ell\in\Z}\widetilde{\eta}_H(\rone+\tfrac k K,\gfour+\tfrac{\ell K} L)  e^{2\pi i   \frac {(j-k)\ell} L}e^{-2\pi i \frac k K\gfour}  \big)\notag \\
                      && \hspace{5cm}\overline{\pi(\rtwo,\gtwo)\varphi(\rone)}  \,   \, \psi(\gfour )\, \overline{\pi(\gthree,\rthree)\varphi(\gfour )}\,d\gfour \,d\rone \big\|_{L^{p}_{1{\otimes}w_1}}^p\,d\rtwo \notag \\
&= & \sum_{j=0}^{L-1}
                     \int_{0}^{\frac {1} K}\iint \big|\int \int \sum_{k\in\Z}\sum_{\ell \in\Z}c_{j-k}\,e^{2\pi i   \frac {(j-k)\ell} L} \widetilde{\eta}_H(\rone+\tfrac k K,\gfour+\tfrac{\ell K} L)\notag
                      \\
                      && \hspace{5cm}e^{-2\pi i \frac k K\gfour} \, \overline{\pi(\rtwo,\gtwo)\varphi(\rone)}  \,   \, \psi(\gfour )\, \overline{\pi(\gthree,\rthree)\varphi(\gfour )}\,d\gfour \,d\rone \,w_1(\rthree)\big|^p\,d\rthree\,d\gthree\,d\rtwo \notag
\\
&\geq&
                \sum_{j=0}^{L-1}
                                \int_{0}^{\frac {1} K}\!\!\!\!
                             \iint_0^{\frac K L }      \!\!\big| \int   \sum_{k\in\Z}  \sum_{\ell\in\Z} c_{j-k}\,
                                        e^{2\pi i\tfrac{(j-k)\ell}{L}}
                                        \notag \\
                                        &&
                                        \hspace{2.4cm}
                                        \int e^{-2\pi i \gfour  \tfrac {k} K}\,
                                             \widetilde{\eta_H}(\rone  +\tfrac k K,\gfour +\ell\tfrac{K}{L})\,
                                                    \overline{\pi(\rtwo,\gtwo )\varphi(\rone  )}
                                            \, \psi(\gfour )\, \overline{\pi(\gthree,\rthree)\varphi(\gfour )}\,\,  d\gfour\,d\rone
                                   \,w_1(\rthree)\big|^p\,d\gthree \, d\rthree  \, d\rtwo
                     \notag\\
                     %\\
&=&  \sum_{j=0}^{L-1}
                                \int_{0}^{\frac {1} K}\!\!\!\!
                                    \iint_0^{\frac K L}\big|   \sum_{k\in\Z}  \sum_{
                                    \ell\in\Z} c_{j-k}\,
                                        e^{2\pi i\tfrac{(j-k)\ell}{L}}
                                        \notag \\
                                        &&
                                        \hspace{2.4cm}
                                        \int\!\!\!\! \int e^{-2\pi i \gfour  \tfrac {k} K}\,
                                             \widetilde{\eta_H}(\rone  +\tfrac k K,\gfour +\ell\tfrac{K}{L})\,
                                                    \overline{\pi(\rtwo,\gtwo )\varphi(\rone  )}
                                            \,  \overline{\pi(\gthree,\rthree)\varphi(\gfour )}\,\,d\rone\,   d\gfour
                                   \,w_1(\rthree)\big|^p\,\,d\gthree \, d\rthree  \, d\rtwo
                     \label{equation:useSupportOf-g}\\
  &=&  \sum_{j=0}^{L-1}
                                \int_{0}^{\frac {1} K}\!\!\!\!
                                    \iint_0^{\frac K L} \big|   \sum_{k\in\Z}  \sum_{
                                    \ell\in\Z} c_{j-k}\,
                                        e^{2\pi i\tfrac{(j-k)\ell}{L}}
                                        \notag \\
                                        &&
                                        \hspace{2.4cm}
                                        V_{\varphi{\otimes}\varphi}  \widetilde{\eta_H}(\rtwo +\tfrac k K, \gthree+\tfrac {\ell K} L,\gtwo ,\rthree+\tfrac k K )
                                            \, e^{2\pi i \frac k K \gtwo}\,e^{2\pi i \frac {\ell k} L}\,e^{2\pi i \rthree \frac {\ell K} L }
                                   \,w_1(\rthree)\big|^p\,\,d\gthree \, d\rthree  \, d\rtwo
                    \notag \\
&=&
                                \int_{0}^{\frac {1} K}\!\!\!\!
                                    \iint_0^{\frac K L} \sum_{j=0}^{L-1}\big|   \sum_{j'=0}^{L}  c_{j-k_{j'}}\,
                                        e^{2\pi i\tfrac{j\ell_{j'}}{L}}
                                        \notag \\
                                        &&
                                        \hspace{2.4cm}
                                        V_{\varphi{\otimes}\varphi}  \widetilde{\eta_H}(\rtwo +\tfrac {k_{j'}} K, \gthree+\tfrac {\ell_{j'} K} L,\gtwo ,\rthree+\tfrac {k_{j'}} K )
                                            \, e^{2\pi i \frac {k_{j'}} K \gtwo}\,e^{2\pi i \rthree \frac {\ell_{j'} K} L }
                                   \,w_1(\rthree)\big|^p\,\,d\gthree \, d\rthree  \, d\rtwo
                     \label{equation:useSupportOf-g2}\\
                     %\\
                    &\asymp&
                                 \int_{0}^{\frac {1} K}\!\!\!\!
                                    \iint_0^{\frac K L}   \sum_{j'=0}^{L}\big|
                                        V_{\varphi{\otimes}\varphi}  \widetilde{\eta_H}(\rtwo +\tfrac {k_{j'}} K, \gthree+\tfrac {\ell_{j'} K} L,\gtwo ,\rthree+\tfrac {k_{j'}} K )
                                            \, e^{2\pi i \frac {k_{j'}} K \gtwo}\,e^{2\pi i \rthree \frac {\ell_{j'} K} L }
                                   \,w_1(\rthree+\tfrac {k_{j'}} K)\big|^p \,d\gthree \, d\rthree  \, d\rtwo
                     \notag\
                     \\
                      &=&
                           \sum_{j'=0}^{L}  \int_{0}^{\frac {1} K}\!\!\!\!
                                    \iint_0^{\frac K L} \big|
                                        V_{\varphi{\otimes}\varphi}  \widetilde{\eta_H}(\rtwo +\tfrac {k_{j'}} K, \gthree+\tfrac {\ell_{j'} K} L,\gtwo ,\rthree )
                                 \,w_1(\rthree)  \big|^p \,d\gthree \, d\rthree  \, d\rtwo
                     \notag
                      \\
                      &=&
                          \sum_{j'=0}^{L}
                                 \int_{\frac {k_{j'}} K}^{\frac {(k_{j'}+1)} K}\!\!\!\!
                                    \iint_{\frac{l_{j'}K} L}^{\frac{(l_{j'}+1)} L} \big|
                                        V_{\varphi\otimes\varphi}
                                        \widetilde\eta( \rtwo,\gthree ,
                                            \gtwo ,\rthree )
                                    \,w_1(\rthree)\big|^p \,d\gthree \, d\rthree  \, d\rtwo\,.
                        \notag
                     \end{eqnarray}
To obtain \eqref{equation:useSupportOf-g2} we used that $V_{\varphi{\otimes}\varphi}\widetilde{\eta}\subseteq [0,1){\times} [0,K)$. Moreover,  we used
\begin{eqnarray}
  &&\hspace{-.5cm}  \int\!\!\!\! \int e^{-2\pi i \gfour  \tfrac {k} K}\,
                                             \widetilde{\eta_H}(\rone +\tfrac k K,\gfour +l\tfrac{K}{L})\,
                                                    \overline{\pi(\rtwo,\gtwo )\varphi(\rone  )}
                                            \,  \overline{\pi(\gthree,\rthree)\varphi(\gfour )}\,\,d\rone\,   d\gfour \notag \\
  &=&  \int\!\!\!\! \int e^{-2\pi i \gfour  \tfrac {k} K}\,
                                             \widetilde{\eta_H}(\rone +\tfrac k K,\gfour +l\tfrac{K}{L})\,
                                                    e^{-2\pi i \rone \gtwo}\varphi(\rone-\rtwo  )
                                            \,  e^{-2\pi i \rthree \gfour}\varphi(\gfour -\gthree)\,\,d\rone  \, d\gfour \notag \\
  &=&  \int\!\!\!\! \int   \widetilde{\eta_H}(\rone ,\gfour)\,
                                                   e^{-2\pi i \gtwo (x-\frac k K)} \varphi(\rone-(\rtwo+\tfrac k K  ))
                                            \, e^{-2\pi i (\gfour-\frac {\ell K} L)  \tfrac {k} K}\, e^{-2\pi i \rthree(\gfour-\frac{\ell K}L)}\,\varphi(\gfour-(\gthree+\tfrac {\ell K} L) )\,\,d\rone  \, d\gfour \notag \\
&=& V_{\varphi{\otimes}\varphi}  \widetilde{\eta_H}(\rtwo +\tfrac k K, \gthree+\tfrac {\ell K} L,\gtwo ,\rthree+\frac k K )
                                            \, e^{2\pi i \frac k K \gtwo}\,e^{2\pi i \frac {\ell k} L}\,e^{2\pi i \rthree \frac {\ell K} L } \,. \notag
\end{eqnarray}

Using the fact that replacing now $\varphi\otimes \varphi$ by any other test functions leads to equivalent norms of the modulation space at hand, we obtain for real valued $\varrho\in S(\R^2)$, $\varrho(t,\nu)=1$ for $[-1,2){\times}[-K,2K)$,

\begin{eqnarray}
&&\hspace{-1cm}\|Hg\|_{M_w^{p,q}} \asymp \|V_\varphi Hg\|_{L^{p,q}}
                  =  \Big\| \Big(\sum_{n\in\Z} \int_{\frac n K}^{\frac {n+1} K}
                     |V_\varphi Hg(u,\gtwo)w_1(u)|^p  \Big)^{\frac 1 p}\Big\|_{L_{w_2}^q} \notag \\
                      &=&
                          \Big\|
                            \Big( \sum_{j'=0}^{L}
                                 \int_{\frac {k_{j'}} K}^{\frac {(k_{j'}+1)} K}\!\!\!\!
                                    \int_{\frac{l_{j'}K} L}^{\frac{(l_{j'}+1)} L}\!\!\!\!\int  \big|
                                        V_{\varphi\otimes\varphi}
                                        \widetilde\eta( u,\gone,
                                            \gtwo,\rone )
                                    w_1(\rone)\big|^p \,d\gone \,d\rone  \, du
                            \Big)^{\frac 1 p}
                        \Big\|_{L_{w_2}^q}\notag \\
                      &=&
                          \Big\|
                            \Big(
                                 \int_0^1\!\!\!\!
                                    \int_0^K\!\!\!\!\int  \big|
                                        V_{\varphi\otimes\varphi}
                                        \widetilde\eta( u,\gone,
                                            \gtwo,\rone )
                                     w_1(\rone)\big|^p \,d\gone \,d\rone  \, du
                            \Big)^{\frac 1 p}
                        \Big\|_{L_{w_2}^q}\notag \\
                      &\asymp&
                        \Big\|
                            \Big(
                                 \int_0^1\!\!\!\!
                                    \int_0^K\!\!\!\!\int  \big|
                                        V_{\varrho}
                                        \widetilde\eta( u,\gone,
                                            \gtwo,\rone )
                                     w_1(\rone)\big|^p \,d\gone \,d\rone  \, du
                            \Big)^{\frac 1 p}
                        \Big\|_{L_{w_2}^q}\notag \\
                      &\geq&
                        \Big\|
                            \Big(
                                 \int_0^1\!\!\!\!
                                    \int_0^K\!\!\!\!\int  \big|
                                        \chi_{[0,1)}(u)\chi_{[0,K)}(\gone)V_{\varrho}
                                        \widetilde\eta( u,\gone,
                                            \gtwo,\rone )
                                     w_1(\rone)\big|^p \,d\gone \,d\rone  \, du
                            \Big)^{\frac 1 p}
                        \Big\|_{L^q(\gtwo)}\notag \\
                        &\geq&
                        \Big\|
                            \Big(
                                \int_0^1\!\!\!\!
                                    \int_0^K\!\!\!\!\int \big|
                                        \chi_{[0,1)}(u)\chi_{[0,K)}(\gone) \mathcal F^s {\widetilde  \eta}(
                                            \gtwo,\rone )
                                    w_1(\rone)\big|^p \,d\gone \,d\rone  \, du
                            \Big)^{\frac 1 p}
                        \Big\|_{L^q(\gtwo)}\notag \\
                      &=&
                        \Big\|
                            \Big(
                              K
                                    \int_0^K \big|
                                          \mathcal F^s {\widetilde  \eta}(
                                            \gtwo,\rone )
                                    w_1(\rone)\big|^p  \,d\rone
                            \Big)^{\frac 1 p}
                        \Big\|_{L_{w_2}^q}=
                        \Big\|
                            \Big(
                              K
                                    \int_0^K \big|
                                        {\widetilde  \sigma}(
                                            \rone ,\gtwo)
                                    w_1(\rone)\big|^p  \,d\rone
                            \Big)^{\frac 1 p}
                        \Big\|_{L_{w_2}^q}\notag \\ &=&
                          \|
                                  \widetilde     \sigma
                        \|_{L_w^{p,q}}\asymp
                          \|
                                       \sigma
                        \|_{L_w^{p,q}}\,.
                     \label{equation:normequiv}
                     \end{eqnarray}

                     To obtain \eqref{equation:normequiv}, we apply a mixed $L^p$-space version of Young's inequality for convolutions, namely, we use that for $\widetilde \varrho(x,\xi)=e^{2\pi i x \xi}\varrho(x,\xi)\in L^1(\R)$, we have $\widetilde \sigma (x,\xi)=\widehat{\widetilde \varrho}\ast \sigma$ and $ \sigma (x,\xi)=\widehat{\overline{\widetilde \varrho}}\ast \widetilde\sigma$ (see Theorem 11.1, \cite{Gro01}).

\subsection{Outline of the proof of Theorem~\ref{thm:main-full3}}

\begin{figure}
%\begin{center}
\setlength{\unitlength}{.5cm}
\hspace{-.6cm}\begin{picture}(13,10)\thicklines
    \put(1,7){$OPW_w^{p,q}(M)$}
    \put(7.7,7.7){$\Phi_g$}
    %\put(6.8,7.7){$C_{(g,\Gamma)}$}
    \put(11,7){$M_w^{p,q}(\R)$}
    %\put(11,7){$l_s^p(\Z^{2d})$}
    \put(5.5,7.2){\vector(1,0){4.7}}
    \put(2.2,1){$l_{\widetilde w}^{p,q}(\Z^{2})$}
    \put(1,4){$D_{\{P_j\}}$}
    \put(3.2,2.3){\vector(0,1){4}}
    \put(5.5,1.3){\vector(1,0){4.7}}
    \put(7.7,1.6){$M$}
    \put(11,1){$l_{\widetilde w}^{p,q}(\Z^{2})$}
    \put(12.2,6.3){\vector(0,-1){4}}
    \put(12.4,4){$C_{(\ga,\,\lllambda\Z^{2d})}$}
\end{picture}\quad \quad \quad \quad
\begin{picture}(13,10)\thicklines
    \put(2,7){$\sum_j c_j P_j$}
    \put(7.7,7.7){$\Phi_g$}
    %\put(6.8,7.7){$C_{(g,\Gamma)}$}
    \put(11,7){$\sum_j c_j P_j g$}
    %\put(11,7){$l_s^p(\Z^{2d})$}
    \put(5.5,7.2){\vector(1,0){4.7}}
    \put(2.2,1){$\left\{ c_j  \right\}_j$}
    \put(1,4){$D_{\{P_j\}}$}
    \put(3.2,2.3){\vector(0,1){4}}
    \put(5.5,1.3){\vector(1,0){4.7}}
    \put(7.7,1.6){$M$}
    \put(11,1){$\left\{\sum c_j \langle P_j g,\pi(\lllambda j')\ga\right\}_{j'}$}
    \put(12.2,6.3){\vector(0,-1){4}}
    \put(12.4,4){$C_{(\ga,\,\lllambda\Z^{2d})}$}
\end{picture}
%\end{center}

\caption{Sketch of the proof of
Theorem~\ref{thm:main-full3}. We choose a structured
operator family $\{P_j\}\subseteq OPW_w^{p,q}(M)$ so that the corresponding
synthesis map $D_{\{P_j\}}:\,\{c_j\}\longrightarrow \sum c_j P_j$
has a bounded left inverse. Note that $C_{(\ga,\,\lllambda\Z^{2d})}$
has a bounded left inverse for $\lllambda < 1$ as well.
Theorem~\ref{theorem:mainresult} shows that for any $g\in
M^\infty(\R)$, the composition $M = C_{(\ga,\,\lllambda\Z^{2d})}{\circ}
\phi_g {\circ} D_{\{P_j\}}$ has {\it no} bounded left inverses. This implies
that $\phi_g: OPW_w^{p,q}(M) \longrightarrow M^{p,q}_w(\R)$ also has {\it no} bounded left inverses.  \label{figure:thirdresult}}
\end{figure}

We omit detailed computations as they would closely resemble computations carried out in \cite{KP06,Pfa08b,PW06,PW06b}. For the interested reader, we suggest to use \cite{PW06b} as a companion when filling in detail.

We shall show that for a measurable subset $M$ with $\vol^-(M)>1$, the operator class $OPW^{p,q}_w(M)$ is not identifiable. In detail, we shall show that  for {\em every} $g\in M^{\infty,\infty}(\R)$, the operator
$$
    \Phi_g: OPW^{p,q}_w(M) \longrightarrow M_w^{p,q}(\R),\ H\mapsto Hg\,,
$$
is {\em not} bounded below, that is, there exists {\it no} $c>0$ for which we have $\|Hg\|_{M^{p,q}_\weight}\geq c\|\sigma_H\|_{L_w^{p,q}}$ for all $H\in OPW^{p,q}_w(M)$.

To this end, choose
$K,\,L\in\N$ and $V_M=
                \bigcup_{j=0}^{L-1} \left(R_{K,L}+(\tfrac {m_j} K  , \tfrac {n_j K} {L}   )
                \right),\
                \ m_j,n_j\in\Z,$ where $R_{K,L}=[0,\frac 1 K)\times[0,\frac K L)$ and where
$(m_j,n_j)\ne(m_{j'},n_{j'})$ if $j\ne j'$, such that
$V_M\subseteq M$ and $\vol(V_M)>1$.  It is sufficient to
show that $OPW^{p,q}_w(V_M)$ is not identifiable as $OPW^{p,q}_w (V_M)\subseteq  OPW^{p,q}_w(M)$.

The proof of Theorem~\ref{thm:main-full3} is also sketched in Figure~\ref{figure:thirdresult}. The proof is based on extensions to results from \cite{Pfa08b,PW06b} which are stated below and which concern the construction of the operator family $\{P_j\}$ in Figure~\ref{figure:thirdresult}.

%%%%%%%%%%%%%%%%%%%%%%%%%%%%%%%%%%%%%%%%%%%%%%%%%%%%%%%%%%%%%%%%%%%
\begin{lemma} \label{lemma:time-frequency-spreading-shift}
Let $P:\mathcal S(\R)\rightarrow \mathcal S'(\R)$. For $p,\,r\in\R$ and $\weight,\,\xi\in\Rh$,
define $\widetilde{P}= M_{\weight}T_{p - r} P
T_{r}M_{\xi-\weight}: \mathcal S(\R)\rightarrow \mathcal S'(\R)$. Then $\eta_{\widetilde{P}}=e^{2\pi i
r\xi}M_{(\weight,r)}\,T_{(p,\xi)}\,\eta_P$.
\end{lemma}

\begin{lemma}\label{lem:elementarybound} Fix
$\lllambda>1$ with $1<\lllambda^4<\tfrac J L$ and $0<\epsilon <1$. Choose
$\eta_1,\eta_2\in \mathcal{S}(\R)$ with values in $[0,1]$,
$$
  \eta_1(t)=
  \begin{cases}
    1 & \text{for}\ t\in [\tfrac{\lllambda-1}{2\lllambda K}, \tfrac{\lllambda+1}{2\lllambda K}) \\
    0 & \text{for}\ t\notin [0,\tfrac 1 K)
  \end{cases}\quad \text{and}
  \ \quad
  \eta_2(\nu)=
  \begin{cases}
    1 & \text{for}\ \nu\in [\tfrac{(\lllambda-1) K}{2\lllambda L}, \tfrac{(\lllambda+1) K}{2\lllambda L}) \\
    0 & \text{for}\ \nu\notin [0, \tfrac K L)
  \end{cases}\ ,
$$
and $|\mathcal F \eta_1 (\xi)|\leq Ce^{-\gamma |\xi|^{1-\epsilon}}$, $|\mathcal F^{-1} \eta_2(x)|\leq Ce^{-\gamma |x|^{1-\epsilon}}$ .
Let $\eta_P=\eta_1\otimes\eta_2$.  Then $\supp\eta_P\subseteq[0,\tfrac 1 K)\times[0,\tfrac K L)=R_{K,L}$
and the operator  $P\in OPW^{1,1}(R_{K,L})$ has the following properties:

\noindent a) The operator family
\begin{eqnarray*}
   \left\{
            M_{\lllambda  K k}\,
            T_{\frac 1 K m -\frac  {\lllambda L}{K}l }\,
            P\,
            T_{\frac {\lllambda L}{ K}l} \,
            M_{\frac{ K}{L}n-\lllambda K k}\,\right\}_{k,l,m,n\in\Z}
\label{equation:OperatorFamily}
\end{eqnarray*}
is an $l^{p,q}_{\widetilde w}$--Riesz basis sequence for $OPW^{p,q}_w(\R^2)$.

\noindent b)  $P\in OPW^{1,1}(R_{K,L})$   and
there exist  functions $d_1,d_2:\R\rightarrow\R_0^+$ with
$$
|Pf(x)|\leq \|f\|_{M^{\infty,\infty}}\,d_1(x) \quad \text{and}\quad
 |\widehat{Pf}(\xi)|\leq \|f\|_{M^{\infty,\infty}}\,d_2(\xi),\quad f\in M^{\infty,\infty}(\R),
$$
and $d_1(x)\leq \widetilde C e^{-\widetilde \gamma |x|^{1-\epsilon}}$, $d_2(\xi)\leq \widetilde Ce^{-\widetilde \gamma |x|^{1-\epsilon}}$ for some $\widetilde C, \widetilde \gamma >0$.

\end{lemma}

\begin{proof} The existence of $\eta_1,\eta_2$ satisfying the hypotheses stated above is established through mollifying characteristic functions.  In fact, using constructions of Gevrey class functions, it has been shown that for $\epsilon,\delta>0$, there exists  $\varphi:\R\rightarrow \R^+_0$ and $C,\gamma>0$ with $\supp \varphi \subseteq [-\delta,\delta]$, $\int\varphi =1$, $\widehat \varphi (\xi)\leq Ce^{-\gamma |\xi|^{1-\epsilon}}$\cite{Hor03,Hor05,DH98}. Note that the restriction to $w$ subexponential in Theorem~\ref{thm:main-full3} is a consequence to the fact that there exist no compactly supported functions whose Fourier transforms decay exponentially (see references in \cite{GP08}).

{\it a)} Due to Lemma
\ref{lemma:time-frequency-spreading-shift},
 \begin{eqnarray}
   \left\{
            M_{\left( \lllambda K k,\frac{\lllambda L}{ K}l\right)}
               \, T_{\left(\frac 1 K m, \frac{K}{L}n\right)}\ \eta_P
          \right\}_{k,l,m,n\in\Z}
\notag
\end{eqnarray}
being an $l^{p,q}_{\widetilde w}$--Riesz basis for its closed linear span in
$M_{1{\otimes}w}^{(1,1),(p,q)}(\R^2)$
implies that
\begin{eqnarray}
   \left\{
            M_{\lllambda  K k}\,
            T_{\frac 1 K m -\frac  {\lllambda L}{K}l }\,
            P\,
            T_{\frac {\lllambda L}{ K}l}\,
            M_{\frac{ K}{L}n-\lllambda K k}\right\}_{k,l,m,n\in\Z}
\notag
\end{eqnarray}
is an $l^{p,q}_{\widetilde w}$-Riesz basis for its closed linear span in $OPW^{p,q}_w(\R^2)$.

{\it b)} As shown in the proof of Lemma 3.4 in \cite{KP06},   we have $
  |Pf(x)|\leq \left|\widehat{\eta}_2(-x)\right|\
       \|f\|_{M^{\infty,\infty}} \|\eta_1\|_{M^{1,1}}$, so we can choose $d_1(x)=\left|\widehat{\eta}_2(-x)\right|
\|\eta_1\|_{M^{1,1}} $.

Similarly, we can compute $ |\widehat{Pf}(\xi)|\leq
       \|f\|_{M^{\infty,\infty}} \| \eta_2(\xi-\,\cdot)\widehat{\eta}_1(\cdot)\|_{M^{1,1}}.$
We claim that $d_2(\xi)= \|
\eta_2(\xi-\,\cdot)\widehat{\eta}_1(\cdot)\|_{M^{1,1}}$  has the postulated subexponential decay. Recall that for $g$  supported on $[a,b]$, we have $\|g\|_{M^{1,1}}\leq c\  \|\widehat g\|_{L^1}$ where $c$ depends only on the support size  $b-a$ (see Lemma~\ref{lemma:compactbounded} and \cite{Oko09}).  As $\eta_2(\xi-\,\cdot)\widehat{\eta}_1(\cdot)$ is compactly supported with uniform support size, we can compute
\begin{eqnarray}
  d_2(\xi)&=&  \|
\eta_2(\xi-\,\cdot)\widehat{\eta}_1(\cdot)\|_{M^{1,1}} \leq c\  \|
\mathcal F^{-1}\big(\eta_2(\xi-\,\cdot)\widehat{\eta}_1(\cdot)\big)\|_{L^{1}} \notag \\
&=&c\ \int \Big| \int \eta_2(\xi-\nu)\widehat{\eta}_1(\nu) e^{2\pi i x \nu} \, d\nu \Big|\, dx=c\ \int \Big|V_{\widetilde \eta_2}\widehat \eta_1 (\xi,-x)  \Big|\,dx
\,,\end{eqnarray}
where $\widetilde \eta_2(\xi)=\eta_2(-\xi)$. As the compact support of $\eta_1,\eta_2$ together with  $|\mathcal F \eta_1 (\xi) | \leq Ce^{-\gamma |\xi|^{1-\epsilon}}$, $|\mathcal F^{-1} \eta_2(x)|\leq Ce^{-\gamma |x|^{1-\epsilon}}$ imply that $\widehat \eta_1,\widetilde \eta_2$ are in the Gelfand--Shilov class $\mathcal S_{1-\epsilon}^{1-\epsilon}$ \cite{GS68}, we apply Proposition 3.12 in \cite{GZ04} to conclude that $V_{\widetilde \eta_2}\widehat \eta_1\in \mathcal S_{1-\epsilon}^{1-\epsilon}$, that is,
\begin{eqnarray*}
  d_2(\xi)&\leq & c\ \int \widetilde C^{-1} e^{-\widetilde \gamma \|(x,y)\|_\infty^{1-\epsilon}}  \,dx \ \leq\ \widetilde {\widetilde C}\,e^{-\widetilde \gamma |\xi|^{1-\epsilon}}
\,.\end{eqnarray*}
\end{proof}

Theorem~\ref{theorem:mainresult} extends the main result in \cite{Pfa08b} to weighted and mixed $l^p$ spaces with subexponential weights. Both results generalize to infinite dimensions the fact that $m\times n$
matrices with $m<n$ have a non--trivial kernel and, therefore,
are not bounded below  as operators acting on $\C^n$.

\begin{theorem}\label{theorem:mainresult}
Let $1\leq p_1,p_2,q_1,q_2 \leq \infty$, $w_1,w_2$ subexponential on $\Z^{2d}$, that is, for some $C,\gamma >0$, $0<\beta<1$, we have
\begin{eqnarray} C^{-1} \, e^{-\gamma \|n\|^\beta_\infty} \leq w_1(n),w_2(n)\leq C\, e^{\gamma \|n\|^\beta_\infty}\,. \label{eqn:subexponential}
\end{eqnarray}
If for  $M=(m_{j'j}):l_{w_1}^{p_1,q_1}(\Z^{2d})\rightarrow l_{w_2}^{p_2,q_2}(\Z^{2d})$, exists
$\lllambda>1$, $K_0>0$, and
$$\rho:\R^+_0\longrightarrow \R^+_0\ \text{ with }\ \rho\leq \widetilde C\, e^{-\widetilde \gamma \|n\|^{\widetilde \beta}_\infty},\quad \widetilde \beta > \beta,$$
with
$$\displaystyle
  |m_{j'j}|\leq \rho (\lllambda \|j'\|_\infty-\|j\|_\infty)
    ,\quad \lllambda \|j'\|_\infty-\|j\|_\infty > K_0,
$$ then $M$ has no bounded left inverses.
\end{theorem}

\begin{proof}
Let $v(n)=C\, e^{\gamma \|n\|^\beta_\infty}$. Note that $l_{w_1}^{p_1,q_1}(\Z^{2d})$ embeds continuously in $l_{1/v}^{\infty,\infty}(\Z^{2d})=l_{1/v}^{\infty}(\Z^{2d})$ and $l_{v}^{1,1}(\Z^{2d})=l_{v}^{1}(\Z^{2d})$ embeds continuously in $l_{w_2}^{p_2,q_2}(\Z^{2d})$. Hence, it suffices to show that for all $\epsilon >0$ exists $x\in l_{1/v}^{\infty}(\Z^{2d})$ with $\|x\|_{l_{1/v}^\infty}=1$ and $\|Mx\|_{l_{v}^1}\leq \epsilon$. For notational simplicity, we replace $2d$ by $D$ in the following.

First, observe that
\eqa{A_{K_1}=
    e^{\gamma K_1^\beta}
        \sum_{K\geq K_1}
            K^{D-1} e^{\gamma K^\beta}
        \sum_{k\geq K}
            k^{D-1}\,  e^{-\widetilde \gamma k^{\widetilde \beta}}
    \to 0\,
    \text{ as } K_1\to \infty . \label{equation:DoubleSumW}
}
Applying the integral criterion for sums, this would follow from
\eqa{
    e^{\gamma K_1^\beta}
        \int_{ K_1}^\infty
            x^{D-1} e^{\gamma x^\beta}
        \int_{x}^\infty
            y^{D-1}\,  e^{-\widetilde \gamma y^{\widetilde \beta}}\, dy \, dx
    \to 0\,
    \text{ as } K_1\to \infty .\label{equation:DoubleIntegralW}
}
For large $x$, a substitution yields
\begin{eqnarray}
   \int_x^\infty y^{D{-}1}\,  e^{{-}\widetilde \gamma y^{\widetilde \beta}}\, dy&{=}& \frac 1 {\widetilde\beta \widetilde\gamma} \int_{\widetilde\gamma x^{\widetilde \beta}}^\infty \big(\tfrac t {\widetilde\gamma}\big)^{\frac 1 {\widetilde\beta} (D{-}1)}\, \big(\tfrac t {\widetilde\gamma} \big)^{\frac {1{-}\widetilde\beta} {\widetilde\beta} (D{-}1)}  e^{{-}t}\, dt
{=} \frac 1 {\widetilde\beta \widetilde\gamma} \widetilde \gamma^{2(D{-}1)(1{-}\frac 1 {\widetilde \beta})} \int_{\widetilde\gamma x^{\widetilde \beta}}^\infty t^{2 \frac {D{-}1}{\widetilde \beta}{{-}}D{+} 2 {-} 1 } e^{{-}t} dt\notag \\
&=& \frac 1 {\widetilde\beta \widetilde\gamma} \widetilde \gamma^{2(D{-}1)(1{-}\frac 1 {\widetilde \beta})}\ \Gamma( 2 \tfrac {D{-}1}{\widetilde \beta}{{-}}D{+} 2,\widetilde \gamma x^{\widetilde \beta})
 \leq \frac 2 {\widetilde\beta \widetilde\gamma} \widetilde \gamma^{2(D{-}1)(1{-}\frac 1 {\widetilde \beta})}   \big(\widetilde \gamma x^{\widetilde \beta}\big)^{ 2 \tfrac {D{-}1}{\widetilde \beta}{-}D{+} 1}\ e^{{-}\widetilde \gamma x^{\widetilde \beta}}
 \notag \\
 &\leq&  \frac 2 {\widetilde\beta} \widetilde \gamma^{D{-}2}\ x^{2  {D{-}1}{-}\widetilde \beta (D{+} 1)}\ e^{{-}\widetilde \gamma x^{\widetilde \beta}} \notag\notag\, ,
\end{eqnarray}
 where $\Gamma$ denotes the upper incomplete Gamma function, and we used $\frac {\Gamma (s,y)}{y^{s-1}e^{-y}} \to 1$ as $y\to \infty$ \cite{AMS92}.

The fact that $\widetilde \beta >\beta$ allows us to estimate for large $x$
$$e^{\gamma x^\beta} e^{-\widetilde \gamma x^{\widetilde \beta}}= e^{-\widetilde \gamma x^{\widetilde \beta}(1- \frac {\gamma} {\widetilde \gamma} x^{\beta-\widetilde \beta})} \leq e^{-\frac 1 2 \widetilde \gamma x^{\widetilde \beta}}\,. $$
Hence, we can repeat the arguments above to the outer integral and the product with $e^{\gamma K_1^\beta}
$ in  \eqref{equation:DoubleIntegralW} to  obtain \eqref{equation:DoubleIntegralW}, and, hence, \eqref{equation:DoubleSumW} holds.

To continue our proof, we fix
$\epsilon>0$ and note that (\ref{equation:DoubleSumW}) provides us
with a $K_1>K_0$ satisfying $A_{K_1} \leq  (2^{2D}D^2\,\widetilde C\, C^2\, e^{\gamma N^\beta})^{-1}\, \epsilon\,.
$

As in \cite{Pfa08b}, set $N=\left\lceil \frac{ \lambda (K_1+1)}{\lambda - 1}
\right\rceil$ and $\widetilde{N}=\lceil \frac{N}{\lambda}
\rceil+K_1.$ Then $\frac{\lambda(K_1+1)}{\lambda - 1}
  \leq N\leq\frac{\lambda(K_1+2)}{\lambda - 1}
$  implies $ \lambda N\geq \lambda K_1 + \lambda +N$ and
$
N\geq
 K_1 + \frac N \lambda +1
    > K_1 + \left\lceil \frac N \lambda\right\rceil=\widetilde N.
 $
Hence, $(2\widetilde{N}+1)^D<(2N+1)^D$ so that the matrix $
\widetilde{M} = ( m_{j'j})_{\|j'\|_\infty\leq
\widetilde{N},\|j\|\leq N}: \C^{(2N+1)^D}\longrightarrow
\C^{(2\widetilde{N}+1)^D}$ has a nontrivial kernel. We now choose
$x\in {l_{1/v}^\infty}(\Z^D)$ with  $\|x\|_{l_{1/v}^\infty}=1$, $x_j=0$ if $\|j\|_\infty> N$, and
and $\widetilde{M}\widetilde{x}=0$ where $\widetilde x$ is $x$ restricted to the set $\{j: \ \|j\|_\infty\leq N\}$ .

By construction we have $(Mx)_{j'}=0$ for
$\|j'\|_\infty\leq \widetilde{N}$. To estimate $(Mx)_{j'}$ for
$\|j'\|_\infty> \widetilde{N}$, we fix $K > K_1$ and one of the
$2D(2 \big(\lceil\frac{N}{\lambda} \rceil+K\big))^{D-1}$ indices
$j'\in\Z^D$ with $\|j'\|_\infty=\lceil\frac{N}{\lambda} \rceil+K$.
We have $\|\lambda j'\|_\infty \geq N+K \lambda$ and $\lambda
\|j'\|_\infty-\|j\|_\infty \geq K\lambda\geq K$ for all $j\in\Z^D$
with $\|j\|_\infty\leq N$. Therefore, using H\"older's inequality for weighted $l^p$-spaces, we obtain
\begin{eqnarray*}
|(Mx)_{j'}|
  &=&    \Big|\sum_{\|j\|_\infty\leq N} m_{j'j}x_j\Big|
  \leq \|x\|_{l_{1/v}^\infty}   \sum_{\|j\|_\infty\leq N} v(j)\, \left|m_{j'j}\right|
    \\
  &\leq& C\, e^{\gamma N^\beta}
  \sum_{\|j\|_\infty\leq N} \rho(\lambda \|j'\|_\infty-\|j\|_\infty)
  \leq C\, e^{\gamma N^\beta} \sum_{\|j\|_\infty\geq K}\rho (\|j\|_\infty)
  =2^DD\, C\, e^{\gamma N^\beta} \sum_{k\geq K}k^{D-1} \rho(k).
\end{eqnarray*}
Next, we  compute
\eqa{ \|Mx\|_{l_v^{1}}
  &=& \sum_{j'\in\Z^D} v(j')\,|(Mx)_{j'}|
  = \sum_{\|j'\|_\infty \geq \lceil \frac{N}{\lambda} \rceil+K_1}
        v(j')\, |(Mx)_{j'}|\notag \\
  &\leq& 2^DD\, C\, e^{\gamma N^\beta}
        \sum_{\|j'\|_\infty \geq \lceil \frac{N}{\lambda} \rceil+K_1} v(j')
        \sum_{k\geq \|j'\|_\infty}k^{D-1} \rho(k)
        \notag\\
  &\leq&2^DD\, C\, e^{\gamma N^\beta}
        \sum_{K\geq \lceil \frac{N}{\lambda} \rceil+K_1}
            2D (2K)^{D-1} C\, e^{\gamma K^\beta}
       \sum_{k\geq K}
            k^{D-1}\, \rho(k)\notag \\
        &\leq&2^{2D}D^2\,\widetilde C\, C^2\, e^{\gamma N^\beta}
        \sum_{K\geq \lceil \frac{N}{\lambda} \rceil+K_1}
            K^{D-1} e^{\gamma K^\beta}
        \sum_{k\geq K}
            k^{D-1}\,  e^{-\widetilde \gamma k^{\widetilde \beta}}\leq \epsilon \notag\,.
}
\end{proof}

%%%%%%%%%%%%%%%%%%%%%%%%%%%%%%%%%%%%%%%%%%%%%%%%%%%%%%%%%%%%%%%%%%%%%%%%%%%%%%%%%%%%%%%%%%%%%%%

Combining the results above, we can now proceed to prove Theorem~\ref{thm:main-full3}.

%%%%%%%%%%%%%%%%%%%%%%%%%%%%%%%%%%%%%%%%%%%%%%%%%%%%%%%%%%%%%%%%%%%%%%%%%%%%%%%%%%%%%%%%%%%%%%%
%\begin{theorem}For $M\subset\R^2$ with $\vol(M)>1$, $\h_M$ is not identifiable. \label{theorem:necessity}
%\end{theorem}
%%%%%%%%%%%%%%%%%%%%%%%%%%%%%%%%%%%%%%%%%%%%%%%%%%%%%%%%%%%%%%%%%%%%%%%%%%%%%%%%%%%%%%%%%%%%%%%

\vspace{.3cm}
{\it Proof of Theorem~\ref{thm:main-full3}}. As $w$ is subexponential, there exists $C,\gamma,\epsilon>0$ with $|w(x,\xi)|\leq C\,e^{\gamma \|(x,\xi)\|_\infty^{1- 2\epsilon}}$. For this $\epsilon>0$ choose $\lllambda$, $\eta_1$, $\eta_2$, $P$, $d_1$, and $d_2$ as in
Lemma~\ref{lem:elementarybound}.

Define the synthesis operator $E:l^{p,q}_w(\Z^2)\rightarrow OPW^{p,q}_w(V_M)\subseteq OPW^{p,q}_w(M)$ as
follows. For $\sigma=\{\sigma_{k,p}\}\in l^{p,q}_{\widetilde w}(\Z^2)$ write
$\sigma_{k,p}=\sigma_{k,lJ+j}$ for $l\in\Z$ and $0\le j<J$ and
define
\begin{eqnarray}\label{eqn:synthesismap}
  E(\sigma) = \sum_{k,l\in Z}\sum_{j=0}^{J-1}\ \sigma_{k,lJ{+}j} \
            M_{ \lllambda  K k}\,
            T_{\frac{1}{K}k_j + \frac  {\lllambda L}{K}l }\,
            P\,
            T_{- \frac  {\lllambda L}{K}l }\,
            M_{\frac{ K}{L}p_j-\lllambda K k}
\end{eqnarray}
with convergence in case $p,q\neq \infty$ and weak-$\ast$ convergence else.
Since
\begin{eqnarray}
  \left\{
            M_{\lllambda  K k}\,
            T_{\frac 1 K m -\frac  {\lllambda L}{K}l }\,
            P\,
            T_{\frac {\lllambda L}{ K}l} \,
            M_{\frac{ K}{L}n-\lllambda K k}\right\}_{k,l,m,n\in\Z}
\notag
\end{eqnarray}
is an $l_{\widetilde w}^{p,q}$--Riesz basis for its closed linear span in $OPW^{p,q}_w(\R^2)$, so is its subset
\begin{eqnarray}
    \left\{
            M_{\lllambda  K k}\,
            T_{\frac{1}{K}k_j +\frac  {\lllambda L}{K}l }\,
            P\,
            T_{- \frac  {\lllambda L}{K}l }\,
            M_{\frac{ K}{L}p_j-\lllambda K k}\right\}_{k,l\in\Z,0\le j<J}
\notag
\end{eqnarray}
 and $E$ is bounded and bounded below .

By Theorem~\ref{thm:GaborEssentails}, the Gabor system
$
  (\ga ,a'\Z{\times}b'\Z)
    =\left\{M_{ka'}T_{lb'} \ga \right\}
$
is an $l^{p,q}_{\widetilde w}$--frame for any  $a',b'>0$ with $a'b'<1$, and we
conclude that the  analysis map given by
\begin{equation}\label{eqn:analysismap}
  C_{\ga }:M_w^{p,q}(\R)  \rightarrow  l_{\widetilde w}^{p,q}(\Z^2), \quad
         f\mapsto
         \left\{
           \langle f, M_{\lllambda^2 K\,k}
            T_{\frac{\lllambda^2  L}{KJ}\,l } \ga \rangle
         \right\}_{k,l}
\end{equation}
is bounded and bounded below  since $\lllambda^2 K\frac{\lllambda^2 L}{KJ}
=\lllambda^4 \frac{L}{J} <1.$

For simplicity of notation, set $\alpha =K$ and $ \beta=\frac L
{KJ} $. Fix $f\in M^{\infty,\infty}(\R)$ and consider the
composition
$$
\begin{array}{ccccccc}
  l^{p,q}_{\widetilde w}(\Z^2)      & \stackrel{E}{\rightarrow} & OPW^{p,q}_w(M)       & \stackrel{\Tg}{\rightarrow} &  M^{p,q}_w(\R)
        & \stackrel{C_{\ga }}{\rightarrow} & l^{p,q}_{\widetilde w}(\Z^2)
  \\
  \sigma         & \mapsto         & E\sigma &  \mapsto   & E\sigma\,g
        & \mapsto         &
        \left\{\, \langle\, E\sigma\,g,\ M_{\lllambda^2\alpha k'}T_{\lllambda^2\beta l'} \,
         \ga  \, \rangle \, \right\}_{k',l'}\, .
\end{array}
$$
It is easily computed that the operator $C_{\ga }{\circ} \Phi_f {\circ} E$ is
represented
--- with respect to the canonical basis $\{\delta(\cdot-n)\}_n$
of $l_{\widetilde w}^{p,q}(\Z^2)$ ---  by the bi-infinite matrix
\begin{eqnarray*}
{\mathcal M}=%\Big(m_{k',l',k,l''}\Big)=
    \Big(m_{k',l', \,k \,,lJ+j}\Big)=
    \Big(   \langle\, \
            M_{\lllambda \alpha k}\,
            T_{\frac{k_j}{\alpha} + \lllambda   \beta l J}\,
            P\,
            T_{- \lllambda \beta l J  }\,
            M_{\frac{  p_j}{\beta J}-\lllambda \alpha k}\,f\, ,\
                M_{\lllambda^2\alpha k'}T_{\lllambda^2 \beta l'} \, \ga \,
            \rangle
    \Big) \, .
\end{eqnarray*}
Setting \begin{eqnarray*}
%\widetilde{d}_1=\sum_{j=0}^{J-1} T_{\frac{m_j}{\alpha} - \lambda
%                    \beta j} d_1
\widetilde{d}_1=\max_{j=0,\dots,J-1} T_{\frac{k_j}{\alpha} - \lambda
                    \beta j} d_1\,,\label{equation:d_1-tilde}
\end{eqnarray*}
we observe
\begin{eqnarray*}
  |m_{k',l',k,lJ+j}|
        &\leq&
          \big\langle\,
            T_{  \lllambda   \beta (l J+j)}
            \big( T_{\frac{k_j}{\alpha} - \lllambda
                    \beta j} \left|P\,
            T_{-\lllambda \beta l J  }\,
            M_{\frac{  p_j}{\beta J}-\lllambda \alpha k}\,f\,\right|\big)\,
            ,
               T_{\lllambda^2\beta l'} \ga \,
            \big\rangle
        \notag\\
        &\leq& \|f\|_{M^{\infty,\infty}} \
          \big\langle\,
            T_{  \lllambda   \beta (l J+j)}
            T_{\frac{k_j}{\alpha} - \lllambda
                    \beta j} d_1\, ,\
               T_{\lllambda^2\beta l'} \ga \,
            \big\rangle
        \notag\\
     %   &\leq& \|f\|_{M^{\infty,\infty}} \
%          \big\langle\,
%            T_{  \lllambda   \beta (l J+j)}
%            \widetilde{d}_1\, ,\
%               T_{\lllambda^2\beta l'} \ga \,
%            \big\rangle
%        \notag\\
      &\leq& \|f\|_{M^{\infty,\infty}} \ (\widetilde{d}_1 \ast \ga ) \, (\lllambda\beta(\lllambda l'-(lJ+j)))\, ,
      \notag
\end{eqnarray*}
and
\begin{eqnarray*}
    |m_{k',l',k,lJ+j}|
       &=&
         \Big|
          \big\langle\,
                             T_{\lllambda \alpha k}\,
            M_{-\frac{k_j}{\alpha} - \lllambda   \beta l J}\,
            \big(P\,
            T_{-\lllambda \beta l J  }\,
            M_{\frac{  p_j}{\beta J}-\lllambda \alpha k}\,f\,\big)\widehat{\ }\,
            ,
                T_{\lllambda^2\alpha k'}M_{-\lllambda^2\beta l'} \ga \,
            \big\rangle
         \Big|
        \notag \\
  &\leq&
          \langle\,
                 T_{\lllambda \alpha k}
            \left|
            \big(P\,
            T_{-\lllambda \beta l J  }\,
            M_{\frac{  p_j}{\beta J}-\lllambda \alpha k}\,f\,\big)\widehat{\ }\ \right|\, ,\
                T_{\lllambda^2\alpha k'} \ga \,
            \rangle
                 \notag \\
  &\leq& \|f\|_{M^{\infty,\infty}} \ (d_2 \ast \ga ) (\lllambda\alpha(\lllambda k'-k))\notag.
\end{eqnarray*}

Observing that for appropriate $\widetilde C, \widetilde \gamma$, we have $\widetilde d_1\ast \ga (x)d_2\ast \ga(\xi)\leq \widetilde C\ e^{- \widetilde \gamma \|(x,\xi)\|_\infty^{1- \epsilon}}$ and ${1- 2\epsilon}<{1- \epsilon}$, allows us to apply Theorem~\ref{theorem:mainresult} to $\mathcal M$. This completes the proof. \hfill $\square$

\section{Outlook}\label{section:outlook}
Recent results in operator identification with relevance to  operator sampling include, for example,  the identification of Multiple Input Multiple Output (MIMO) channels \cite{Pfa08} and an extension of operator sampling to irregular sampling sets \cite{HP09}.
Some more fundamental questions concerning sampling and identification of operator Paley--Wiener spaces are still open. In the following we describe two such questions.

\subsubsection*{Unbounded spreading domains with small Lebesgue measure.}

The extension of Theorem~\ref{thm:main-full2} to $OPW^{p,q}_w(M)$ with $M$ unbounded but with Lebesgue measure less than one remains open.
The following observations encourage tackling this question:
\begin{enumerate}

  \item Multiplication operators with not necessarily bandlimited symbol in $L^2(\R^2)$ are clearly identifiable with identifier $g=\chi_\R\in M^{\infty,\infty}(\R)$. Note that the characteristic function $\chi_\R$  is the weak-$\ast$ limit of $T\sum_{n\in\Z} \delta_{nT}$ as $T\to 0$. Hence, the space $OPW^{\infty,\infty}\big(\{0\}{\times}\widehat \R \big)$ is identifiable.

  \item Time--invariant operators with not necessarily compactly supported $L^2(\R^2)$ impulse response are identifiable with identifier $\delta$ which is  the weak-$\ast$ limit of $\sum_{n\in\Z} \delta_{nT}$ as $T\to \infty$. Consequently, the space $OPW^{\infty,\infty}\big(\R{\times}\{0\} \big)$ is identifiable.

       \item In \cite{KP06} it is shown that  $OPW(M)$ is identifiable if $M$ is a possibly unbounded  fundamental domain a lattice $\Lambda$ in $\R^2$ with $\Lambda$ having density less than or equal to one. This result covers, for example, $OPW\big(\big\{(t,\nu): \ t{\geq} -1,\ \nu{<}1,\ 2^{-(t+1)}{\leq} \nu {\leq} 2^{-t}\big\}\big)$ as the unbounded set $\big\{(t,\nu): \ t{\geq} -1,\ \nu{<}1,\ 2^{-(t+1)}{\leq} \nu {\leq} 2^{-t}\big\}$ is a fundamental domain of $\Z^2$.

\end{enumerate}

The natural approach to construct identifiers for $OPW(A)$ as weak-$\ast$ limit of identifiers $g_N$ for $OPW(A\cap [-N,N]{\times} [-N,N])$ is difficult as the constants implied by $\asymp$ in \eqref{eqn:mainresult2item1} depend in a non-trivial matter on $g_N=\sum_n c_{n,N} \delta_{x_{n,N}}$, $N\in\N$, if the sequences $\{ c_{n,N} \}$ are not constant.

\subsubsection*{Generalizations to higher dimensions.}

As mentioned in Section~\ref{section:MainResults2}, our proof  of Theorem~\ref{thm:main-full2} hinges on the existence of identifiers in an analogous setup where the locally compact Abelian group $\R$ is replaced by an appropriate finite cyclic group of prime order $\Z_p$  \cite{KPR08,LPW05}. In fact, generalizing Theorem~\ref{thm:main-full2}, to operators acting on $L^2(\R^d)$ would be possible if  the conclusions of Theorem~\ref{prop:generallinear} hold for sufficiently many composites $n$ taking the place of prime $p$.
In fact, in \cite{KPR08}, we ask the following
\begin{question}Is it true that for all $L\in\N$ exists $c\in\C^L$ so that the vectors  $\pi(k,\ell)c$, $k,\ell \in 0,\ldots, L-1$, defined by $(\pi(k,\ell)c)_j=c_{j-k}\,e^{2\pi i \frac {j\ell}L}$, $k,\ell=0,\ldots, L-1$, are in general linear position.
\end{question}

\section{Acknowledgement}
Foremost, I would like to thank David Walnut as only his contributions and guidance allowed for the joint development of this sampling theory for operators. I would like to thank Werner Kozek for introducing me to the topic of operator identification and John Benedetto, Hans Feichtinger, Niklas Grip, Karlheinz Gr\"ochenig, Yoon Mi Hong, Kurt Jetter, Felix Krahmer, Onur Oktay, and Peter Rashkov for enriching discussions on the time--frequency analysis of operators.

\bibliography{../Bibliography/gabor_goetz}
%\bibliography{gabor_goetz}
\bibliographystyle{alpha}

%%%%%%%%%%%%%%%%%%%%%%%%%%%%%%%%%%%%%%%%%%%%%%%%%%
%
% The Bibliographic References
%

\end{document}